\documentclass[graybox]{svmult}

\usepackage{mathptmx}       
\usepackage{helvet}         
\usepackage{courier}        
\usepackage{type1cm}        
\usepackage{makeidx}         
\usepackage{graphicx}        
\usepackage{multicol}        
\usepackage[bottom]{footmisc}

\usepackage[section]{algorithm}
\usepackage{algorithmic}

\makeindex             

\usepackage{amsmath}
\usepackage{amssymb}
\usepackage{xcolor}

\DeclareMathOperator{\supp}{supp}

\DeclareMathOperator{\diag}{diag}

\begin{document}

\title*{A compressive spectral collocation method\\ for the diffusion equation under the\\ restricted isometry property}
\titlerunning{A compressive spectral collocation method for the diffusion equation}

\author{Simone Brugiapaglia}
\institute{Simone Brugiapaglia \at Simon Fraser University, 8888 University Dr, Burnaby, BC, V5A 1S6, Canada.\\\email{simone\_brugiapaglia@sfu.ca}}
%
%
\maketitle



\abstract*{We propose a compressive spectral collocation method for the numerical approximation of Partial Differential Equations (PDEs). The approach is based on a spectral Sturm-Liouville approximation of the solution and on the collocation of the PDE in strong form at randomized points, by taking advantage of the compressive sensing principle. The proposed approach makes use of a number of collocation points substantially less than the number of basis functions when the solution to recover is sparse or compressible. Focusing on the case of the diffusion equation, we prove that, under suitable assumptions on the diffusion coefficient, the matrix associated with the compressive spectral collocation approach satisfies the restricted isometry property of compressive sensing  with high probability. Moreover, we demonstrate the ability of the proposed method to reduce the computational cost associated with the corresponding full spectral collocation approach while preserving good accuracy through numerical illustrations.}

\abstract{We propose a compressive spectral collocation method for the numerical approximation of Partial Differential Equations (PDEs). The approach is based on a spectral Sturm-Liouville approximation of the solution and on the collocation of the PDE in strong form at randomized points, by taking advantage of the compressive sensing principle. The proposed approach makes use of a number of collocation points substantially less than the number of basis functions when the solution to recover is sparse or compressible. Focusing on the case of the diffusion equation, we prove that, under suitable assumptions on the diffusion coefficient, the matrix associated with the compressive spectral collocation approach satisfies the restricted isometry property of compressive sensing  with high probability. Moreover, we demonstrate the ability of the proposed method to reduce the computational cost associated with the corresponding full spectral collocation approach while preserving good accuracy through numerical illustrations.}


\section{Introduction}

Compressive Sensing (CS) is a mathematical principle introduced in 2006 that allows for the efficient measurement and reconstruction of sparse and compressible signals. Its success is now established in the signal processing community and its wide range of applications include medical imaging, computational biology, geophysical data analysis, compressive radar, remote sensing, and machine learning. More recently, CS has also started attracting more and more attention in scientific computing and numerical analysis, in particular, in the fields of numerical methods for Partial Differential Equations (PDEs), high-dimensional function approximation, and  uncertainty quantification. 

In this paper, we present a novel technique for the numerical solution of PDEs based on CS. The proposed approach, called \textit{compressive spectral collocation} takes advantage the CS principle in the context of spectral collocation methods. Its constitutive elements are: (i) Sturm-Liouville spectral approximation, (ii) randomized collocation, and (iii) greedy sparse recovery. In order to make the presentation easier and the theoretical analysis of the method accessible, we focus on the case of a stationary diffusion equation over a tensor product domain with homogeneous boundary conditions. 

\subsection{Main contributions}

We propose a novel numerical method for PDEs, called compressive spectral collocation, focusing on the case of a stationary diffusion equation over a tensor product domain with homogeneous boundary conditions. The approach leverages the CS principle by randomizing the choice of the collocation points and by promoting sparse solutions with respect to a Sturm-Liouville basis, which are recovered via the greedy algorithm orthogonal matching pursuit. 

Our main contributions can be summarized as follows:

\begin{enumerate}
\item In Algorithm~\ref{alg:CSC}, we present a rigorous formulation of the compressive spectral collocation approach for the diffusion equation.

\item In Theorem~\ref{thm:RIP_CSC}, we prove that the matrix associated with the compressive spectral collocation approach satisfies the restricted isometry property of CS under suitable assumptions on the diffusion coefficient. 

\item In Section~\ref{sec:numerics}, we demonstrate numerically that the compressive spectral collocation approach is able to recover sparse solutions with higher accuracy and lower computational cost than the corresponding ``full'' spectral collocation method. Moreover, in the case of compressible solutions, we show that the compressive approach is computationally less expensive than the full one while maintaining a good level of accuracy. 
\end{enumerate}


Before outlining the structure of the paper, we review the literature about CS-based methods in numerical analysis, placing particular emphasis on numerical methods for PDEs.

\subsection{Literature review}

CS was proposed in 2006 by the pioneering works of Donoho \cite{donoho2006compressed}, Cand\`{e}s, Romberg, and Tao \cite{candes2006robust} and has triggered an impressive amount of work since then. 

The very first attempt to apply CS to the numerical approximation of a PDEs can be found in \cite{JOKAR2010452}. The authors propose a Galerkin discretization of the Poisson problem, where the trial and test spaces are composed by piecewise linear finite elements. The technique is deterministic and relies on the successive refinement of the solution on different hierarchical levels and on a suitable error estimator. Recovery is based on $\ell^1$-minimization.

The CS principle has then been applied to Petrov-Galerkin discretizations of advection-diffusion-reaction equations via the \textsf{COmpRessed SolvING} method (in short, \textsf{CORSING}), proposed in  \cite{brugiapaglia2015compressed}. The method employs  Fourier-type trial functions and wavelet-like test functions (or vice versa) and the dimensionality of the discretization is reduced by randomly subsampling the test space. The theoretical analysis of the method in the infinite-dimensional setting has been carried out in \cite{brugiapaglia2018theoretical}. The \textsf{CORSING} method has also been applied to the two-dimensional Stokes' equation in \cite{brugiapaglia2016compressed}. 

Numerical methods for PDEs based on $\ell^1$ minimization can be considered the ancestors of CS-based methods for PDEs. Lavery conducted some pioneering studies on finite differences for the inviscid Burgers' equation \cite{lavery1988nonoscillatory} and on finite volumes dicretizations for steady scalar conservation laws \cite{lavery1989solution}. More recently, similar techniques have been analyzed for transport and Hamilton-Jacobi equations \cite{guermond2004finite,guermond2009optimal}.
Moreover, some works considered sparsity-promoting spectral schemes for time-dependent multiscale problems based on soft thresholding \cite{mackey2014compressive,schaeffer2013sparse} or on the sparse Fourier transform \cite{daubechies2007sparse}.

On a different but related note, there has been a lot of research activity around CS-based methods for the uncertainty quantification of PDEs with random inputs \cite{bouchot2017multi,doostan2011non,mathelin2012compressed,peng2014weighted,rauhut2017compressive,yang2013reweighted}. In these works, the CS principle is combined with Polynomial Chaos in order to approximate a quantity of interest of the solution map of the PDE. Being very smooth for a wide family of operator equations, this map can be approximated by a sparse combination of orthogonal polynomials and the CS principle employed to lessen the curse of dimensionality. 

Finally, it is worth mentioning recent works where CS is employed to learn the governing equations of a dynamical systems given time-varying measurements \cite{tran2017exact}
and to solve inverse problems in PDEs \cite{alberti2017infinite}.

\subsection{Outline of the paper}

The paper is organized as follows. 

In Section~\ref{sec:CS}, we recall some of the main elements of CS, of particular interest in our context. We place more emphasis on greedy recovery via orthogonal matching pursuit and on recovery guarantees based on the restricted isometry property. 

Equipped with the CS fundamentals, we present the compressive spectral collocation method in Section~\ref{sec:method}, focusing on the case of a homogeneous stationary diffusion equation.  

Section~\ref{sec:theory} deals with the theoretical analysis of the method. We prove that the matrix associated with the compressive spectral collocation approach satisfies the restricted isometry property with high probability under suitable conditions on the diffusion coefficient. Moreover, we discuss the implications of the restricted isometry property for the recovery error analysis of the method.

In Section~\ref{sec:numerics}, we illustrate some numerical results for the two-dimensional diffusion equation. We assess the performance of the compressive spectral collocation approach when recovering sparse and compressible solutions. Moreover, we compare it with the corresponding ``full'' spectral collocation approach, demonstrating the computational advantages associated with the proposed strategy. 

Conclusions and future directions are finally discussed in Section~\ref{sec:conclusions}.

\section{Elements of compressive sensing}
\label{sec:CS}

We introduce some elements of CS that will be useful to define the compressive spectral collocation approach. Our presentation is based on a very special selection of topics. For a comprehensive introduction to CS, we refer the reader to \cite{foucart2013mathematical}.

CS deals with the problem of measuring a sparse or compressible signal by using the minimum amount of linear, nonadaptive observations, and of reconstructing it via convex optimization techniques (such as $\ell^1$ minimization and its variants), greedy algorithms, or thresholding techniques. 

Here, we focus on CS with greedy recovery via orthogonal matching pursuit. After introducing this setting in Section~\ref{sec:sensing_recovery}, we recall some theoretical results based on the restricted isometry property in Section~\ref{sec:recovery}.

\subsection{Compressive sensing and greedy recovery}
\label{sec:sensing_recovery}

Let us consider a vector $x \in \mathbb{R}^N$ (often called ``signal''). We restrict the presentation to the real case, even though the theory can be generalized to the complex case. We collect $m$ linear nonadaptive measurements of $x$ into a vector $b \in \mathbb{R}^m$, i.e.\
\begin{equation}
\label{eq:CSsystem}
A x = b.
\end{equation}
The matrix $A \in \mathbb{R}^{m \times N}$ is called the sensing matrix and $m \ll N$. The problem of finding $x$ given $b$ is clearly ill-posed since the linear system  \eqref{eq:CSsystem} is highly underdetermined. In order to regularize this inverse problem, the \emph{a priori} information assumed on $x$ is \emph{sparsity} or \emph{compressibility}. 

A vector is said to be $s$-sparse if it has at most $s$ nonzero entries. More in general, $x$ is said to be compressible if, for some $p \geq 1$, its best $s$-term approximation error (with respect to some $\ell^p$ norm) $\sigma_s(x)_p$ decays quickly in $s$, where $\sigma_s(x)_p$ is defined as
$$
\sigma_s(x)_p:= \inf\{ \|x-v\|_p : v \in \mathbb{R}^N, \|v\|_0 \leq s\},
$$
with
$$
\|v\|_0 := |\supp(v)|, \quad \supp(v) := \{j \in [N]: v_j \neq 0\}, \quad \forall v \in \mathbb{R}^N,
$$
and where we have employed the notation
$$
[n] := \{1,\ldots,n\}, \quad \forall n \in \mathbb{N}.
$$
Notice that if the signal $x$ is $s$-sparse, then $\sigma_k(x)_p = 0$, for every $k \geq s$ and $p \geq 1$. 

A plethora of recovery strategies is available in order to find sparse or compressible solutions to the linear system \eqref{eq:CSsystem}. In this paper, we focus on the greedy algorithm \emph{Orthogonal Matching Pursuit} (OMP)  (see \cite{temlyakov2003nonlinear} and references therein), outlined in Algorithm~\ref{alg:OMP}.

\begin{algorithm}
\normalsize
\textbf{Inputs:}
\begin{itemize}
\item $A\in \mathbb{R}^{m \times N}$: sensing matrix, with $\ell^2$-normalized columns;
\item $b \in \mathbb{R}^m$: vector of measurements;
\item $K \in \mathbb{N}$: number of iterations.
\end{itemize}
\textbf{Orthogonal Matching Pursuit:}
\begin{enumerate}
\item Let $\hat{x}_0 = 0$ and $S_0 = \supp(\hat{x}_0) = \emptyset$;
\item For $k = 1,\ldots,K$, repeat the following steps:
\begin{enumerate}
\item Find  $\displaystyle j_k  = \arg\max_{j \in [N]} |(A^T(A\hat{x}_{k-1} - b))_j|$;
\item Define $S_k   = S_{k-1} \cup \{j_k\}$;
\item Compute $\displaystyle\hat{x}_{k}  = \arg\min_{v \in \mathbb{R}^{N}} \|A v - b\|_2 \text{ s.t. } \supp(v) \subseteq S_k$.
\end{enumerate}
\end{enumerate}
\textbf{Output:}
\begin{itemize}
\item $\hat{x}_K \in \mathbb{R}^N$: $K$-sparse approximate solution to \eqref{eq:CSsystem}.
\end{itemize}
\caption{\label{alg:OMP}Orthogonal Matching Pursuit (OMP)}
\end{algorithm}

OMP iteratively constructs a sequence of $k$-sparse vectors $\hat{x}_k$ that approximately solve \eqref{eq:CSsystem}, with  $k = 1, \ldots, K$, by adding at most one new entry to the support at each iteration. During the $k$-th iteration, OMP seeks the column of $A$ mostly correlated with the residual associated with the previous approximation $\hat{x}_{k-1}$. Then, the support is enlarged by adding the corresponding index, and the $k$-th approximation $\hat{x}_{k}$ is computed by solving an $m \times k$ least-squares problem. Observe that the least-square problem solved to compute $\hat{x}_k$ is overdetermined if $K \leq m$, which is usually the case in practice.

\subsection{Recovery guarantees based on the restricted isometry property}
\label{sec:recovery}

In order to quantify the approximation error associated with the OMP solution, we present some theoretical results based on the restricted isometry property, which has by now become a standard tool in CS.

\begin{definition} 
A matrix $A \in \mathbb{R}^{m \times N}$ is said to satisfy the \textit{restricted isometry property of order $s$ and constant $0 < \delta  <1$} if
\begin{equation}
\label{eq:RIP}
(1-\delta)\|v\|_2^2 \leq \|A v\|_2^2 \leq (1+\delta)\|v\|_2^2, \quad \forall v \in \mathbb{R}^N, \; \|v\|_0 \leq s.
\end{equation}
The smallest $0 < \delta < 1$ such that \eqref{eq:RIP} holds is referred to as the \textit{$s$-th restricted isometry constant} of $A$ and it is denoted by $\delta_s(A)$.

\end{definition}

Intuitively, the restricted isometry property requires the map $x \mapsto Ax$ to approximately preserve distances when its action is restricted to the set of $s$-sparse vectors, up to a distortion factor $\delta$. Computing $\delta_s(A)$ given $A$ is not computationally feasible in general since it implies a search over all the $N \choose s$ subsets of $[N]$ of cardinality $s$. However, what makes this tool extremely useful in practice is the fact that it is possible to show that certain classes of random matrices satisfy the restricted isometry property with high probability. 

The following theorem gives sufficient conditions for a matrix $A \in \mathbb{R}^{m \times N}$ built by independently selecting $m$ random rows according to a suitable probability density from a ``tall'' matrix $B \in \mathbb{R}^{M \times N}$ in order to satisfy the restricted isometry property (up to a suitable diagonal preconditioning). These conditions depend on the spectrum of the Gram matrix $B^TB$ and on the so-called \textit{local coherence} of $B$, i.e., the vector whose entries are
$$
\max_{j \in [N]}(B_{qj})^2, \quad \forall q \in [M].
$$ 
The proof of this result can be found in \cite[Theorem 1.21]{brugiapaglia2016compressed}. Let us note that this is an extension of the restricted isometry property analysis based on the local coherence for bounded orthonormal systems proposed in \cite{krahmer2014stable}, where the orthonormality condition is relaxed. 
\begin{theorem}
\label{thm:RIP_nonisometry}
Consider $B \in \mathbb{R}^{M \times N}$,  with $M \geq N$, and suppose that there exist two constants $0 < r \leq R < + \infty$ such that the minimum and maximum eigenvalues of $B^TB$ satisfy 
$$
0 < r \leq \lambda_{\min}(B^TB) \leq \lambda_{\max}(B^TB) \leq R. 
$$
Moreover, assume that there exists a vector $\nu \in \mathbb{R}^M$ such that the local coherence of $B$ is bounded from above as follows:
$$
\max_{j \in [N]}(B_{qj})^2 \leq \nu_q, \quad \forall q \in [M].
$$
Then, for every $1-\frac{r}{R} < \delta < 1$, there exists a universal constant $c>0$ such that, provided
$$
m \geq \widetilde{c} \, s \ln^3(s)\ln(N),
$$
and $s \geq  \widetilde{c} \, \ln(N )$, where
$$
\widetilde{c} = c \,\max\left(\frac{\|\nu\|_1}{R},1\right)\left(\delta-\left(1-\frac{r}{R}\right)\right)^{-2},
$$
the following holds.

Let us draw $\tau_1,\ldots,\tau_m$ i.i.d.\ from $[M]$ distributed according to the probability density 
$$
p = \frac{\nu}{\|\nu\|_1} \in \mathbb{R}^M,
$$
and define $A \in \mathbb{R}^{m \times N}$ and $D \in \mathbb{R}^{m \times m}$ as
\begin{equation}
\label{eq:def_A_from_B}
A_{i,j} =B_{\tau_i,j},\quad\forall i\in [m], \; \forall j \in [N], \quad  D= \diag\left(\left(\frac{1}{\sqrt{mRp_i}}\right)_{i \in[m]}\right), \quad 
\end{equation}
where $\diag (v)$ denotes the matrix having the entries of $v$ on the main diagonal and zeros elsewhere. Then, the $s$-th restriced isometry constant of $DA$ satisfies 
$$
\delta_s(DA) \leq \delta,
$$
with probability at least $1-N^{-\ln^3(s)}$.

\end{theorem}

The restricted isometry property is a sufficient condition to show that the vector $\hat{x}_K$ computed by $K$ iterations of OMP is a good approximation to $x$. In particular, a suitable upper bound on the $26s$-th restricted isometry property constant is sufficient for OMP to reach the accuracy of the best $s$-term approximation error up to a universal multiplicative constant using $K =24s$ iterations. The following theorem is a direct consequence of \cite[Theorem~6.25]{foucart2013mathematical}. The recovery error analysis of OMP based on the restricted isometry property was originally proposed in \cite{zhang2011sparse}. 
\begin{theorem}
\label{thm:OMP_recovery}
Let $A\in \mathbb{R}^{m \times N}$ with $\ell^2$-normalized columns such that
\begin{equation}
\delta_{26s}(A) < \frac{1}{6}.
\end{equation} 
Then, there exists a universal constant $C >0$ such that for every $x \in \mathbb{R}^N$ and $b \in \mathbb{R}^m$ such that  \eqref{eq:CSsystem} holds, the vector $ \hat{x}_{K}$ computed by $K = 24s$ iterations of OMP (Algorithm~\ref{alg:OMP}) satisfies  
\begin{equation}
\label{eq:OMP_recovery_l2}
\|x-\hat{x}_K\|_2 \leq C \,\frac{\sigma_s(x)_1}{\sqrt{s}}.
\end{equation}
\end{theorem}

This type of recovery error estimate is called ``uniform'' since it holds for every signal $x\in \mathbb{R}^N$. It is worth noticing that when $x$ is $s$-sparse, it is recovered \textit{exactly} since $\sigma_s(x)_1 = 0$. Moreover, let us observe that a more general version of this theorem holds in the case of noisy measurements, i.e., when $y = Ax + e$, where $e \in \mathbb{R}^m$ is a noise vector corrupting the measurements. In that case, an additive term directly proportional to $\|e\|_2$ appears on the left-hand side of \eqref{eq:OMP_recovery_l2} (see \cite[Theorem~6.25]{foucart2013mathematical}).

\section{Compressive spectral collocation}
\label{sec:method}

We are now in a position to introduce the compressive spectral collocation method. Let us consider the following diffusion equation in strong form: 
\begin{equation}
\label{eq:diffusion_strong}
\begin{cases}
-\nabla \cdot (\eta \nabla u)  = F, & \text{in } \Omega,\\
u = 0, & \text{on } \partial \Omega,
\end{cases}
\end{equation}
where $\Omega = (0,1)^d$ is the physical domain, $u \in C^2(\overline{\Omega})$ is the unknown solution, the function $\eta : \overline{\Omega} \to \mathbb{R}$, with $\eta \in C^1(\overline{\Omega})$ and $\eta(x) \geq \eta_{\min} >0$ for every $x \in \overline{\Omega}$, is the diffusion coefficient, and  $F \in C^0(\overline{\Omega})$ is the forcing term. We also consider the dimension $d$ to be of moderate size.

We will define the compressive spectral collocation approach in two steps. First, we describe the Sturm-Liouville basis, the collocation grid employed, and the corresponding ``full'' spectral collocation method in Section~\ref{sec:full}. Then, in Section~\ref{sec:compressive}, we define the compressive approach, outlined in Algorithm~\ref{alg:CSC}.

\subsection{The spectral basis and the collocation grid}
\label{sec:full}
We discretize equation  \eqref{eq:diffusion_strong} by using a  spectral collocation method based on a Sturm-Liouville basis (for a comprehensive introduction to spectral methods, we refer the reader to \cite{canuto2012spectral,gottlieb1977numerical}). In particular, let us consider the functions
\begin{equation}
\label{eq:def_xi_j}
\xi_j(z) := \frac{2^{d/2}}{\pi^2 \|j\|_2^2}\cdot
\prod_{k = 1}^d \sin(\pi j_k z_k), \quad \forall x \in \overline{\Omega}, \quad \forall j \in \mathbb{N}^d.
\end{equation}
The system $\{\xi_j\}_{j \in \mathbb{N}^d}$ is formed by eigenvectors of the Laplace operator with homogeneous Dirichlet boundary conditions, normalized such that 
$$
\|\Delta \xi_j\|_{L^2(\Omega)} = 1, \quad \forall j \in \mathbb{N}^d.
$$
In fact, the system $\{\Delta\xi_j\}_{j \in \mathbb{N}^d}$ is an orthonormal basis for $L^2(\Omega)$ with respect to the standard inner product $\int_\Omega uv$. 
Expanding a function with respect to this basis (up to normalization) corresponds to the so-called ``modified Fourier series expansion''. The coefficients' decay rate of the modified Fourier series expansion of a function is related to its Sobolev regularity and to suitable boundary conditions involving its derivatives. Here, we will assume the solution $u$ to be regular enough to guarantee the compressibility of its coefficients and, consequently, to enable the application of the CS principle. For more details on the approximation properties of univariate and multivariate modified Fourier series expansions and on their usage in spectral methods for PDEs, we refer the reader to \cite{adcock2009univariate,adcock2010multivariate,iserles2008high,iserles2009high}.

Let us now truncate the multi-index set $\mathbb{N}^d$ by using the tensor product multi-index space of order $n \in \mathbb{N}$, i.e. 
$$
[n]^d \subseteq \mathbb{N}^d,
$$ 
of cardinality $N := n^d$. Given a truncation level $n$, we rescale of the basis functions and define
\begin{equation}
\label{eq:def_psi_j}
\psi_j(z) := 
\frac{\xi_j(z)}{(n+1)^{d/2}} , \quad \forall x \in \overline{\Omega}, \quad \forall j \in [n]^d.
\end{equation}
(The normalizations chosen in \eqref{eq:def_xi_j} and in \eqref{eq:def_psi_j} will turn out to be crucial in order to guarantee the restricted isometry property.)

As a collocation grid, we consider a tensorial grid of uniform step $h = 1/(N+1)$ over $\Omega$, defined as
$$
t_q := \frac{q}{n+1}, \quad \forall q \in [n]^d.
$$ 
Notice that we do not need collocation points on  $\partial \Omega$ since the functions $\{\psi_j\}_{j \in [n]^d}$ already satisfy the homogeneous boundary conditions. 

The resulting ``full'' spectral collocation discretization of \eqref{eq:diffusion_strong} is given by
\begin{equation}
\label{eq:full_SC_system}
B x^{\textnormal{full}} = c,
\end{equation}
where
\begin{equation}
\label{eq:def_B_c}
B_{qj} = [\nabla \cdot (\eta \nabla \psi_j)](t_q), 
\quad c_q = F(t_q), \quad \forall q, j  \in [n]^d,
\end{equation}
and where we are implicitly assuming some ordering for multi-indices in $[n]^d$ (e.g., the lexicographic ordering). Given a solution $x^{\text{full}}$ to the system \eqref{eq:full_SC_system}, the full spectral approximation to $u$ is defined as
\begin{equation}
u^{\textnormal{full}} = \sum_{j \in [n]^d} x^{\textnormal{full}}_j \psi_j.
\end{equation}

\subsection{The compressive approach}
\label{sec:compressive}

Before describing the compressive approach, let us explain the rationale behind the normalizations adopted in \eqref{eq:def_xi_j} and \eqref{eq:def_psi_j} by considering for a moment the simple case of the Poisson equation. The normalization chosen for the system $\{\psi_j\}_{j \in [n]^d}$ ensures that
$$
\text{if }\quad \eta(x) = 1, \; \forall x \in \overline{\Omega}, \quad \text{ then }\quad B = \underbrace{S_n \otimes \cdots \otimes S_n}_{d \text{ times}},
$$
where $\otimes$ denotes the matrix Kronecker product and where $S_n \in \mathbb{R}^{n \times n}$ is the matrix associated with the discrete sine transform, defined as
\begin{equation}
\label{eq:defSn}
(S_n)_{ij} = \sqrt{\frac{2}{n+1}}\sin\left(\frac{ij\pi}{n+1}\right), \quad \forall i,j \in [n].
\end{equation}
In particular, the full spectral collocation matrix $B$ is orthogonal, i.e.\ it satisfies 
\begin{equation}
\label{eq:Poisson_orthogonal}
B^TB = I,
\end{equation} where $I$ is the identity matrix, because $S_n$ is orthogonal. Moreover, the local coherence of the matrix $B$ satisfies the upper bound
\begin{equation}
\label{eq:local_coherence_Poisson}
\max_{j \in [n]^d}(B_{qj})^2 
\leq \left(\frac{2}{(n+1)}\right)^d 
\leq \frac{2^d}{N}, \quad \forall q \in [n]^d.
\end{equation}
Therefore, in view of Theorem~\ref{thm:RIP_nonisometry}, drawing $m$ indices independently distributed according to the uniform measure over $[n]^d$, i.e.,
$$
\tau_1,\ldots,\tau_m 
\quad \text{ i.i.d. with }\quad  
\mathbb{P}\{\tau_i = q\} = \frac{1}{N}, \quad \forall q \in [n]^d, \; \forall i \in[m],
$$
it is natural to define the resulting compressive spectral collocation discretization as
\begin{equation}
\label{eq:CSC_system}
A x = b,
\end{equation}
where the matrix $A \in \mathbb{R}^{m \times N}$ and $b \in\mathbb{R}^m$ are defined as 
\begin{equation}
\label{eq:def_A_b_CSC}
A_{ij} = \sqrt{\frac{N}{m}} \, B_{\tau_i,j}, 
\quad b_i = \sqrt{\frac{N}{m}} \, c_{\tau_i}, \quad \forall i \in[m], \; \forall j \in [n]^d.
\end{equation}
The normalization by a factor $\sqrt{N/m}$ is done in order to ensure the restricted isometry property for $A$ (see Theorem~\ref{thm:RIP_nonisometry}). In particular, we observe that since the probability density is uniform, we have $D = \sqrt{N/m} \cdot I$ in \eqref{eq:def_A_from_B}.

The compressive spectral collocation solution is then computed by applying OMP in order to find a sparse solution to \eqref{eq:CSC_system}, up to normalizing the columns of $A$ with respect to the $\ell^2$ norm. The proposed method is summarized in Algorithm~\ref{alg:CSC}.
\begin{algorithm}
\normalsize
\textbf{Inputs:}
\begin{itemize}
\item $n\in \mathbb{N}$: order of the tensor product multi-index space $[n]^d$; 
\item $m \in \mathbb{N}$: number of randomized collocation points;
\item $K \in \mathbb{K}$: number of OMP iterations.
\end{itemize}
\textbf{Compressive spectral collocation:}
\begin{enumerate}
\item Draw $\tau_1,\ldots,\tau_m$ i.i.d.\ uniformly at random from  $[n]^d$;
\item Build $A \in \mathbb{R}^{m \times N}$ and $b \in \mathbb{R}^m$ according to \eqref{eq:def_A_b_CSC};
\item Build $M = \diag\left((\|a_j\|_2)_{j \in[n]^d}\right) \in \mathbb{R}^{N \times N}$, where $a_j$ is the $j$-th column of $A$;
\item Define $\widetilde{A} = A M^{-1} \in \mathbb{R}^{m \times N}$;
\item Compute $\hat{x}_K \in \mathbb{R}^N$ using OMP (Algorithm~\ref{alg:OMP}) with inputs $\widetilde{A}$, $b$, and $K$;
\item Define $\hat{x} = M \hat{x}_K$;
\item Define $\hat{u} = \displaystyle\sum_{j \in [n]^d} \hat{x}_j \psi_j$, with $\{\psi_j\}_{j \in [n]^d}$ given by \eqref{eq:def_psi_j}.
\end{enumerate}
\textbf{Output:}
\begin{itemize}
\item $\hat{u} \in C^{\infty}(\overline{\Omega})$: Compressive  spectral approximation to  \eqref{eq:diffusion_strong}.
\end{itemize}
\caption{\label{alg:CSC}Compressive spectral collocation}
\end{algorithm}

At least three questions naturally arise at this point:
\begin{enumerate}
\item[(i)] The compressive spectral collocation method looks tailored to the Poisson equation. Does this method work for nonconstant diffusion coefficients?  \label{q:nonconstant}
\item[(ii)] How to choose the input parameters $n$, $m$, and $K$ in Algorithm~\ref{alg:CSC}?  \label{q:param}
\item [(iii)] What are the benefits (if any) of the compressive approach with respect to the full one? \label{q:comparison}
\end{enumerate}
Answering to these questions will be the objective of the next two sections. In particular, Section~\ref{sec:theory} will focus on questions (i) and (ii). Applying the theory of CS introduced in Section~\ref{sec:CS}, we will give a sufficient condition on $\eta$ that implies a positive answer to (i) and propose a recipe for (ii). In Section~\ref{sec:numerics}, we will deal with question (iii) by showing the benefits of the compressive approach with respect to the full one through a numerical illustration. 


\section{Theoretical analysis}
\label{sec:theory}

In the previous section, we have proposed the compressive spectral collocation method for the diffusion equation \eqref{eq:diffusion_strong}, summarized in Algorithm~\ref{alg:CSC}. In this section, we see that, given $n \in \mathbb{N}$, in order to recover the best $s$-term approximation error to $x^{\textnormal{full}}$ (up to a universal constant)  with high probability, it is sufficient to choose the number of collocation points and the iterations of OMP such that
\begin{equation}
\label{eq:sufficient_m_K}
m \geq C 2^ds \ln^3(s) \ln(N) \quad \text{and} \quad
K = 24s,
\end{equation}
where $C>0$ is a universal constant independent of $s$, $n$, and $d$. This shows that for $s \ll N$, the number of collocation points to employ is substantially less than the dimension of the approximation space $N$. In particular, it scales linearly with respect to the target sparsity $s$, up to logarithmic factors. The main ingredients of the theoretical analysis are Theorem~\ref{thm:RIP_nonisometry} and Theorem~\ref{thm:OMP_recovery}. 

\subsection{Restricted isometry property}
 
Let us first consider the case of the Poisson Problem, where $\eta(z) = 1$ for every $z \in \overline{\Omega}$. We have seen that in this case the full spectral collocation matrix $B$ defined in \eqref{eq:def_B_c} is orthogonal (recall \eqref{eq:Poisson_orthogonal}) and that its local coherence satisfies the upper bound \eqref{eq:local_coherence_Poisson}. Therefore, a direct application of Theorem~\ref{thm:RIP_nonisometry} with $r = R = 1$, yields the following restricted isometry result.
\begin{theorem}
\label{thm:RIP_CSC_Poisson}
Let $d,s,N \in \mathbb{N}$, with $s \leq N$. Then, there exists a universal constant $c >0$ such that the following holds. For the Poisson equation, the full spectral collocation matrix $B \in \mathbb{R}^{N \times N}$ defined by \eqref{eq:def_B_c} is orthogonal and the corresponding compressive spectral collocation matrix $A \in \mathbb{R}^{m \times N}$ defined by \eqref{eq:def_A_b_CSC} has the restricted isometry property of order $s$ and constant $\delta$ with probability at least $1-N^{-\ln^s(s)}$, provided that
\begin{equation}
\label{eq:m_rate_Poisson}
m \geq c \, 2^d \delta^{-2} s \ln^3(s)\ln(N),
\end{equation}
and $s \geq c \, \delta^{-2} \ln(N)$.
\end{theorem}

Let us now consider the case of a nonconstant coefficient $\eta$. In this case, $B$ is not necessarily orthogonal and, in order to apply Theorem~\ref{thm:RIP_nonisometry}, we need to estimate the minimum and maximum eigenvalue of the Gram matrix $B^TB$ and to find a suitble upper bound to the local coherence of $B$. Using this strategy, in the next theorem we give sufficient conditions on the diffusion coefficient $\eta$ able to guarantee the restricted isometry property for the compressive spectral collocation matrix $A$ with high probability. 
\begin{theorem} 
\label{thm:RIP_CSC}
Let $d,s,N \in \mathbb{N}$ with $s \leq N$, and $\eta \in C^1(\overline{\Omega})$ satisfying the following conditions:
\begin{align}
\label{eq:cond_eta_1}
& \eta_{\min} := \min_{z \in \overline{\Omega}} \eta(z)> 0,\\
& \|\eta\|_{L^\infty(\Omega)}\sum_{k=1}^d\|(\nabla\eta)_k\|_{L^\infty(\Omega)} < \frac{\pi}{2} \eta_{\min}^2. \label{eq:cond_eta_2}
\end{align}
Then, the full spectral collocation matrix $B$ defined by \eqref{eq:def_B_c} satisfies
\begin{equation}
\label{eq:bound_spectrum_BTB}
r \leq \lambda_{\min}(B^T B) \leq \lambda_{\max}(B^T B) \leq R,
\end{equation}
where
\begin{align}
\label{eq:def_r}
r & := \eta_{\min}^2 - \frac{2}{\pi} \|\eta\|_{L^\infty(\Omega)}\sum_{k = 1}^d \|(\nabla\eta)_k\|_{L^\infty(\Omega)},\\
\label{eq:def_R}
R & := \left(\|\eta\|_{L^\infty(\Omega)} + \frac{1}{\pi}\sum_{k = 1}^d\|(\nabla\eta)_k\|_{L^\infty(\Omega)}\right)^2.
\end{align}
Moreover, given $1-r/R < \delta < 1$ and provided
$$
m \geq \widetilde{C} \, s  \ln^3(s) \ln(N),
$$
and $s\geq \widetilde{C} \ln(N)$, where
$$
\widetilde{C} = C \cdot 2^d \left(\delta - \left(1-\frac{r}{R}\right)\right)^{-2},
$$
the matrix $A/\sqrt{R}$, where $A$ is defined as in \eqref{eq:def_A_b_CSC}  satisfies the restricted isometry property of order $s$ and constant $\delta$ with probability at least $1-N^{-\ln^3(s)}$.
\end{theorem}
\begin{proof}
Let us consider the matrices $S_n \in \mathbb{R}^{n \times n}$ defined as in \eqref{eq:defSn} and $C_n \in \mathbb{R}^{n \times n}$ as
$$
(C_n)_{ij} := \sqrt{\frac{2}{n+1}} \cos\left(\frac{\pi ij}{n+1}\right), \quad \forall i,j \in [N].
$$
Using basic trigonometric formulas and Lagrange's trigonometric inequality, it is not difficult to show that
\begin{align}
\label{eq:property_SN}
S_n^TS_n &= I, \\
\label{eq:property_CN}
C_n^TC_n &= I - \frac{2}{n+1} Q_n,
\end{align}
where $I \in \mathbb{R}^{n \times n}$ is the identity matrix and  $Q_n \in \mathbb{R}^{n \times n}$ is a checkerboard-structured matrix whose entries are 1 on the diagonals of even order and 0 on the diagonals of odd order, namely
$$
(Q_n)_{ij} = \frac{1 - (-1)^{i+j+1}}{2}, \quad \forall i,j \in [n].
$$
As already observed, $S_n$ is orthogonal. On the other hand, $C_n$ is ``almost orthogonal'', up to the corrective term $-\frac{2}{n+1}Q_n$. Now, using the chain rule 
\begin{equation}
\label{eq:expression_B_entries}
-\nabla\cdot(\eta \nabla \psi_j)(t_q) 
= -\eta(t_q) \Delta \psi_j(t_q) - \nabla \eta(t_q) \cdot \nabla \psi_j(t_q), \quad \forall q,j\in[n]^d,
\end{equation}
we see that the full spectral collocation matrix $B$ defined in \eqref{eq:def_B_c} has the form 
$$
B = B_1 + B_2,
$$
where
\begin{align*}
B_1 & = -D_0 \bigotimes_{k = 1}^d S_n,\\
B_2 & = -\sum_{k = 1}^d D_k \left(\bigotimes_{\ell = 1}^{k-1}S_n \otimes (C_n E) \otimes \bigotimes_{\ell = k+1}^d S_n\right) F.
\end{align*}
and
\begin{align*}
D_0 & = \diag\left((\eta(t_q))_{q \in [n]^d}\right) \in \mathbb{R}^{N \times N},\\
D_k & = \diag\left(((\nabla\eta)_k(t_q))_{q \in [n]^d}\right) \in \mathbb{R}^{N \times N}, \quad \forall k \in [d],\\
E & = \diag\left((\pi j)_{j \in[n]}\right) \in \mathbb{R}^{n \times n},\\
F & = \diag\left(\left(\frac{1}{\pi^2\|j\|_2^2}\right)_{j \in[n]^d}\right) \in \mathbb{R}^{N \times N}.
\end{align*}
In particular, we have
$$
\|B v\|_2^2 = \|B_1 v\|_2^2 + \|B_2 v\|_2^2 + 2 v^T B_1^T B_2 v, 
\quad \forall v \in \mathbb{R}^N.
$$ 
Now, we estimate the three terms in the right-hand side separately. Recalling \eqref{eq:property_SN}-\eqref{eq:property_CN}, the first term can be estimated as 
\begin{equation}
\label{eq:bound_B2u}
\eta_{\min} \|v\|_2 \leq \|B_2 v\|_2  \leq  \|\eta\|_{L^{\infty}(\Omega)} \|v\|_2, \quad \forall v \in \mathbb{R}^N.
\end{equation}
As for the second term, we have
$$
\|B_1 v\|_2 
\leq \sum_{k = 1}^d\|(\nabla\eta)_k\|_{L^\infty(\Omega)} \left\|\left(\bigotimes_{\ell = 1}^{k-1}S_n \otimes (C_n E) \otimes \bigotimes_{\ell = k+1}^d S_n\right) F v\right\|_2.
$$
Recalling properties \eqref{eq:property_SN}-\eqref{eq:property_CN}, noticing that $Q_n$ is positive semidefinite, using standard properties of the Kronecker product, and defining $w := F v$, we see that
\begin{align*}
&\left\|\left(\bigotimes_{\ell = 1}^{k-1}S_n \otimes (C_n E) \otimes \bigotimes_{\ell = k+1}^d S_n\right) F v\right\|_2^2 \\ 
& \quad = w^T \left(\bigotimes_{\ell = 1}^{k-1}S_n^TS_n \otimes (E C_n^TC_n E) \otimes \bigotimes_{\ell = k+1}^d S_n^TS_n\right) w\\
& \quad  = w^T \left(\bigotimes_{\ell = 1}^{k-1} I_n \otimes (E C_n^TC_n E) \otimes \bigotimes_{\ell = k+1}^d I_n\right) w\\
& \quad  = w^T \left(\bigotimes_{\ell = 1}^{k-1} I_n \otimes E^2 \otimes \bigotimes_{\ell = k+1}^d I_n\right) w
-\frac{2}{n+1} w^T \left(\bigotimes_{\ell = 1}^{k-1} I_n \otimes (E Q_n E) \otimes \bigotimes_{\ell = k+1}^d I_n\right) w\\
& \quad\leq \left\|\diag\left(\left(\frac{\pi j_k}{\pi^2 \|j\|_2^2}\right)_{j \in [n]^d}\right) v\right\|_2^2
\leq \frac{1}{\pi^2} \|v\|_2^2.
\end{align*}
As a result, we obtain
\begin{equation}
\label{eq:bound_B1u}
0 \leq \|B_2 v\|_2 \leq \frac{1}{\pi}\left(\sum_{k = 1}^d\|(\nabla\eta)_k\|_{L^\infty(\Omega)}\right)\|v\|_2, \quad \forall v \in \mathbb{R}^N.
\end{equation}

Finally, using the Chauchy-Schwartz inequality and combining the inequalities  \eqref{eq:bound_B2u} and \eqref{eq:bound_B1u}, the third term can be estimated as
\begin{equation}
\label{eq:bound_u*B1*B2u}
|v^T B_1^TB_2 v| 
\leq \frac{1}{\pi}\|\eta\|_{L^\infty(\Omega)}\sum_{k = 1}^d\|(\nabla\eta)_k\|_{L^\infty(\Omega)} \|v\|_2^2.
\end{equation}
Finally, combining \eqref{eq:bound_B2u}, \eqref{eq:bound_B1u}, and \eqref{eq:bound_u*B1*B2u} yields the spectral bound \eqref{eq:bound_spectrum_BTB} under sufficient conditions \eqref{eq:cond_eta_1}-\eqref{eq:cond_eta_2} on the diffusion coefficient $\eta$.

The last step is the local coherence upper bound. Recalling \eqref{eq:expression_B_entries} and the definition \eqref{eq:def_R} of $R$, we see that
\begin{align*}
\max_{j \in[n]^d} (B_{qj})^2
& = \max_{j \in [n]^d} \left(-\eta(t_q)\Delta\psi_j(t_q) + \nabla\eta(t_q)\cdot\nabla \psi_j(t_q)\right)^2 \\
& \leq  \left(\frac{2}{n+1}\right)^d \max_{j \in[n]^d}\left(|\eta(t_q)|  + \sum_{k = 1}^d |(\nabla\eta)_k(t_q)| \frac{\pi j_k}{\pi^2 \|j\|_2^2} \right)^2
 \leq \frac{2^d R}{N} =: \nu_q.
\end{align*}
This choice of the local coherence upper bound $\nu$ yields  $\|\nu\|_1 \leq 2^d R$ and $p_q = \nu_q/\|\nu\|_1 = 1/N$, for every $q \in [N]$. Finally, a direct application of Theorem~\ref{thm:RIP_nonisometry} completes the proof. \hfill $\square$
\end{proof}


Let us take a closer look to the sufficient condition \eqref{eq:cond_eta_2} on the diffusion coefficient $\eta$. First of all, it is homogeneous in $\eta$, as it is natural to be expected. Moreover, this condition becomes more and more restrictive as the dimension $d$ increases, which is another tangible effect of the curse of dimensionality. Nevertheless, the following example shows that \eqref{eq:cond_eta_2} can be satisfied in practice.

\begin{example}
Let us consider an affine diffusion coefficient of the form 
$$
\eta(z) = 1 + w^Tz, \quad \forall z \in\overline{\Omega},
$$
where $w \in \mathbb{R}^d$ with $w \geq 0$. In this case,  \eqref{eq:cond_eta_2} is equivalent to 
$$
\|w\|_1 < \frac{1}{2}\left(\sqrt{1+2\pi}-1\right)\approx 0.85.
$$ 
As $d$ gets larger, the above condition becomes more and more restrictive. One possible way to mitigate the effect of $d$ on this condition is by requiring the gradient $w = \nabla \eta$ to be sparse.\hfill $\blacksquare$ 

\end{example}

We conjecture that condition  \eqref{eq:cond_eta_2} is suboptimal and that it could be improved. How to make it less restrictive is an object of future investigation.
Equipped with a restricted isometry property result for the compressive spectral collocation matrix $A$, we can now discuss the recovery guarantees of the proposed approach.

\subsection{Recovery guarantees (discussion)}
\label{sec:recovery_discussion}

In view of Theorem~\ref{thm:OMP_recovery}, the restricted isometry property is a sufficient condition for OMP to recover the best $s$-term approximation error to a given signal up to a universal constant in $K = 24s$ iterations. In this section, we discuss the implications of this result for the compressive spectral collocation approach. A fully rigorous analysis of the recovery guarantees is beyond the objectives of this paper and is left to future work.

In order to combine the restricted isometry result (Theorem~\ref{thm:RIP_CSC}) with the OMP recovery result (Theorem~\ref{thm:OMP_recovery}), one has to take into account the effect of the $\ell^2$-normalization of the columns of $A$ onto the restricted isometry constant. Let $\widetilde{A}$ be the normalized version of $A$, as defined in Algorithm~\ref{alg:CSC}. Then, it is not difficult to show that
$$
\delta_s(\widetilde{A}) 
\leq \frac{2\delta_s(A)}{1-\delta_s(A)}, \quad \forall s\in \mathbb{N}, \; s \leq N.
$$
 In particular, the condition $\delta_{26s}(\widetilde{A}) < 1/6$ required to apply Theorem~\ref{thm:OMP_recovery} and ensure the recovery via OMP, is implied by $\delta_{26s}(A) < 1/13$. 
 
Due to the additional constraint $\delta > 1-r/R$ required by Theorem~\ref{thm:RIP_CSC}, we see that in order to be able to choose $\delta < 1/13$, we need 
\begin{equation}
\label{eq:cond_OMP_recovery}
1-\frac{r}{R} < \frac{1}{13} \; \Longrightarrow \; \frac{r}{R} > \frac{12}{13} \approx 0.92,
\end{equation}
where $r$ and $R$ are defined as in \eqref{eq:def_r}-\eqref{eq:def_R}.

Now, let us notice that a solution $x^{\text{full}}$ to the full system \eqref{eq:full_SC_system} is also a solution to the compressive system \eqref{eq:CSC_system}. Using Theorem~\ref{thm:OMP_recovery} and the fact that $\{\Delta\xi_j\}_{j \in \mathbb{N}^d}$ is orthonormal in $L^2(\Omega)$, we can estimate the error between the full spectral approximation $u^{\textnormal{full}}$ and the compressive spectral approximation $\hat{u}$ computed by choosing $m$ and $K$  as in \eqref{eq:sufficient_m_K} as
\begin{equation}
\label{eq:CSC-full_error}
\|\Delta (u^{\textnormal{full}} - \hat{u})\|_{L^2(\Omega)} 
= \frac{\|x^{\textnormal{full}} - \hat{x}\|_2}{(n+1)^{d/2}} 
\leq C \cdot \frac{\sigma_s(x^{\textnormal{full}})_1}{(n+1)^{d/2}\sqrt{s}},
\end{equation}
where $C>0$ is a universal constant. ($C$ depends on the universal constant of Theorem~\ref{thm:OMP_recovery} and on $\sqrt{1+\delta}$, due to the $\ell^2$-normalization of the columns of $A$. Moreover, notice that we can fix, e.g., $\delta = 1/14 < 1/13$). 
It is worth observing that when $x^{\textnormal{full}}$ is $s$-sparse, the compressive spectral collocation recovers the coefficients of $u^{\textnormal{full}}$ exactly.


\begin{remark}
Combining \eqref{eq:CSC-full_error} with the triangle and the Poincar\'e inequalities yields
\begin{align*}
|u-\hat{u}|_{H^1(\Omega)} 
& \leq |u-u^{\textnormal{full}}|_{H^1(\Omega)} + |u^{\textnormal{full}}-\hat{u}|_{H^1(\Omega)}\\
& \leq |u-u^{\textnormal{full}}|_{H^1(\Omega)} + C_{\Omega}\cdot \|\Delta(u^{\textnormal{full}}-\hat{u})\|_{L^2(\Omega)}\\
& \leq |u-u^{\textnormal{full}}|_{H^1(\Omega)} + C_{\Omega}\cdot C \cdot \frac{\sigma_s(x^{\textnormal{full}})_1}{(n+1)^{d/2}\sqrt{s}},
\end{align*}
where $C_\Omega$ is the Poincar\'e constant of $\Omega$.
In this way, an estimate of $|u-u^{\textnormal{full}}|_{H^1(\Omega)}$ can be converted to an estimate of $|u-\hat{u}|_{H^1(\Omega)}$. The error term $|u-u^{\text{full}}|_{H^1(\Omega)}$ can be studied using the tools in \cite[Section 6.4.2]{canuto2012spectral}. These tools allow to compare $|u-u^{\text{full}}|_{H^1(\Omega)}$ to the best linear approximation error of the solution $u$ with respect to the basis $\{\psi_j\}_{j\in[n]^d}$ in the $H^1(\Omega)$-seminorm. In turn, the best linear approximation error can be estimated by assuming enough regularity of $u$ with respect to standard or mixed Sobolev norms when fulfilling suitable boundary conditions (see \cite[Lemma 3.4 and Lemma 3.5]{adcock2010multivariate}). 
\hfill $\blacksquare$
\end{remark}

\section{Numerical experiments}
\label{sec:numerics}

We conclude by illustrating some numerical experiments that show the robustness of the spectral collocation method described in Algorithm~\ref{alg:CSC} for the numerical solution of the diffusion equation \eqref{eq:diffusion_strong} when the solution is sparse or compressible. The experiments demonstrate that the compressive approach is able to outperform the full one both from the accuracy and the efficiency viewpoints when the solution is sparse. When the solution is compressible, the compressive method can reduce the computational cost of the full method while preserving good accuracy. 

We underline that the comparison is made without using fast transforms, which could considerably accelerate performance of both methods. 

Given the order $n$ of the ambient multi-index set $[n]^d$ and a target sparsity $s \in \mathbb{N}$, in all the numerical experiments we define the number of collocation points and of OMP iterations as
\begin{equation}
\label{eq:choice_m_K}
m = \lceil 2 s \ln(N) \rceil \quad \text{and} \quad K = s,
\end{equation}
numerically showing that the sufficient condition \eqref{eq:sufficient_m_K} is rather pessimistic in practice. Moreover, we focus on a two-dimensional diffusion equation with nonconstant coefficient 
\begin{equation}
\label{eq:nonconstant_coefficient}
\eta(z) = 1+\frac14(z_1+z_2), \quad \forall z \in \overline{\Omega},
\end{equation}
satisfying condition \eqref{eq:cond_eta_2}. 

All the numerical experiments have been performed in \textsc{Matlab$^\text{\textregistered}$} R2017b version 9.3 64-bit on a MacBook Pro equipped with a 3 GHz Intel Core i7 processor and with 8 GB DDR3 RAM. We have employed the OMP implementation provided by the \textsc{Matlab$^\text{\textregistered}$} package \textsf{OMP-Box v10} \cite{rubinstein2008efficient}. 

\subsection{Recovery of sparse solutions} 

We start by comparing the full and the compressive spectral collocation approaches for the recovery of sparse solutions. 

Given $s,n \in \mathbb{N}$ with $s \leq n^2=:N$, we consider $s$-sparse solutions $x$ randomly generated as follows. First, we draw $s$ indices from $[N]$ uniformly at random. Then, we fill the corresponding entries with $s$ independent realizations of a standard Gaussian variable $N(0,1)$. This is implemented in \textsc{Matlab$^\text{\textregistered}$} using the commands \texttt{randperm} and \texttt{randn}, respectively. For each randomly-generated vector $x$, we run the full and the compressive methods 5 times. The recovery error of the full and of the compressive solution is measured using the relative discrete $\ell^2$-error of the coefficients, namely,
$$
\frac{\|x^{\text{full}}-x\|_2}{\|x\|_2}\quad \text{and} \quad \frac{\|\hat{x}-x\|_2}{\|x\|_2}.
$$ 

The results are shown in Fig.~\ref{figure1}, where we plot the relative error as a function of the computational cost for $n = 32$ (corresponding to $N =  1024$) and $s = 2,4,8,16,32$. Recalling \eqref{eq:choice_m_K}, these values correspond to $m = 28, 56, 111, 222, 444$.
\begin{figure}[t]
\centering
\begin{tabular}{ccc}
\includegraphics[height = 3.7cm]{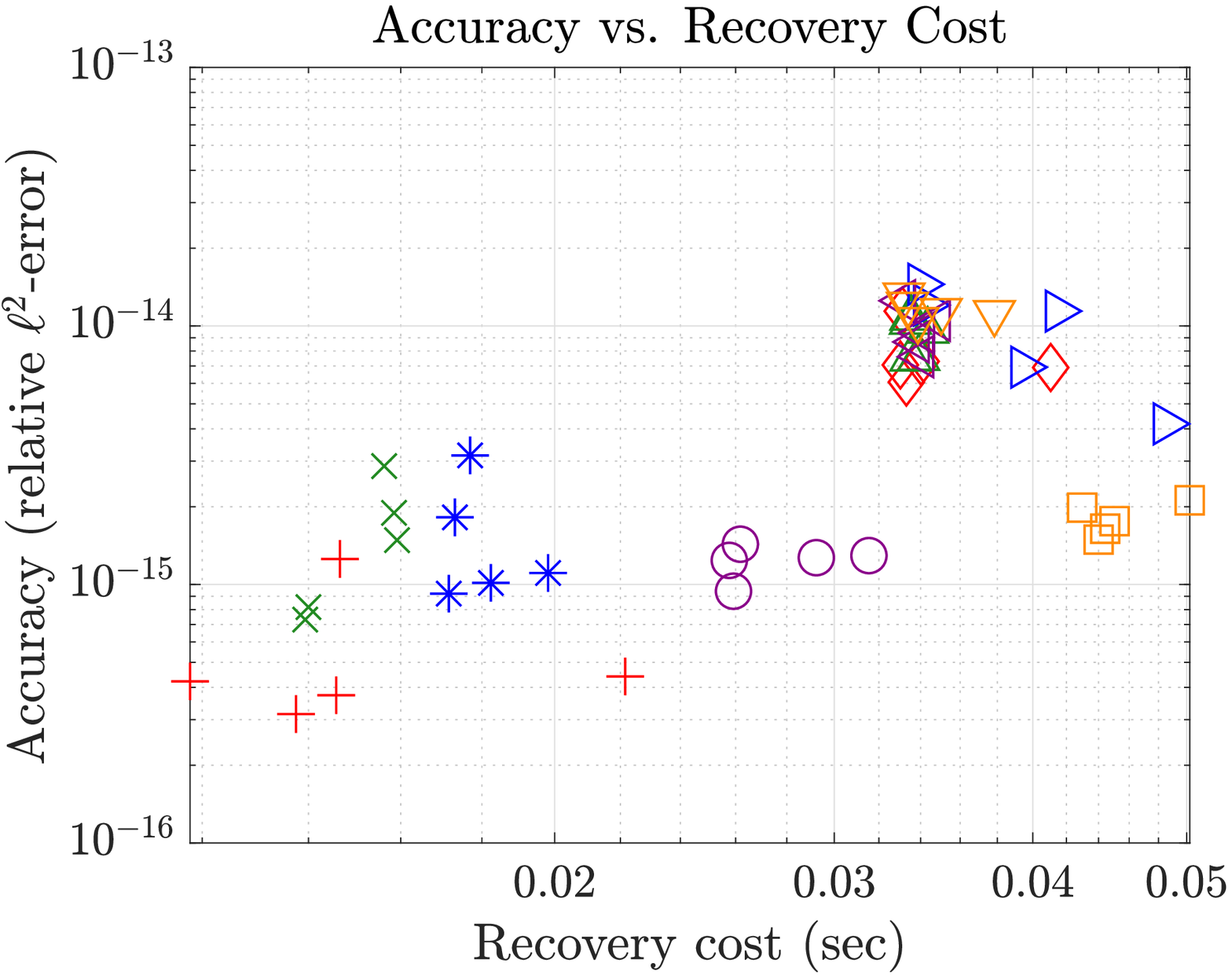}&
\includegraphics[height = 3.7cm]{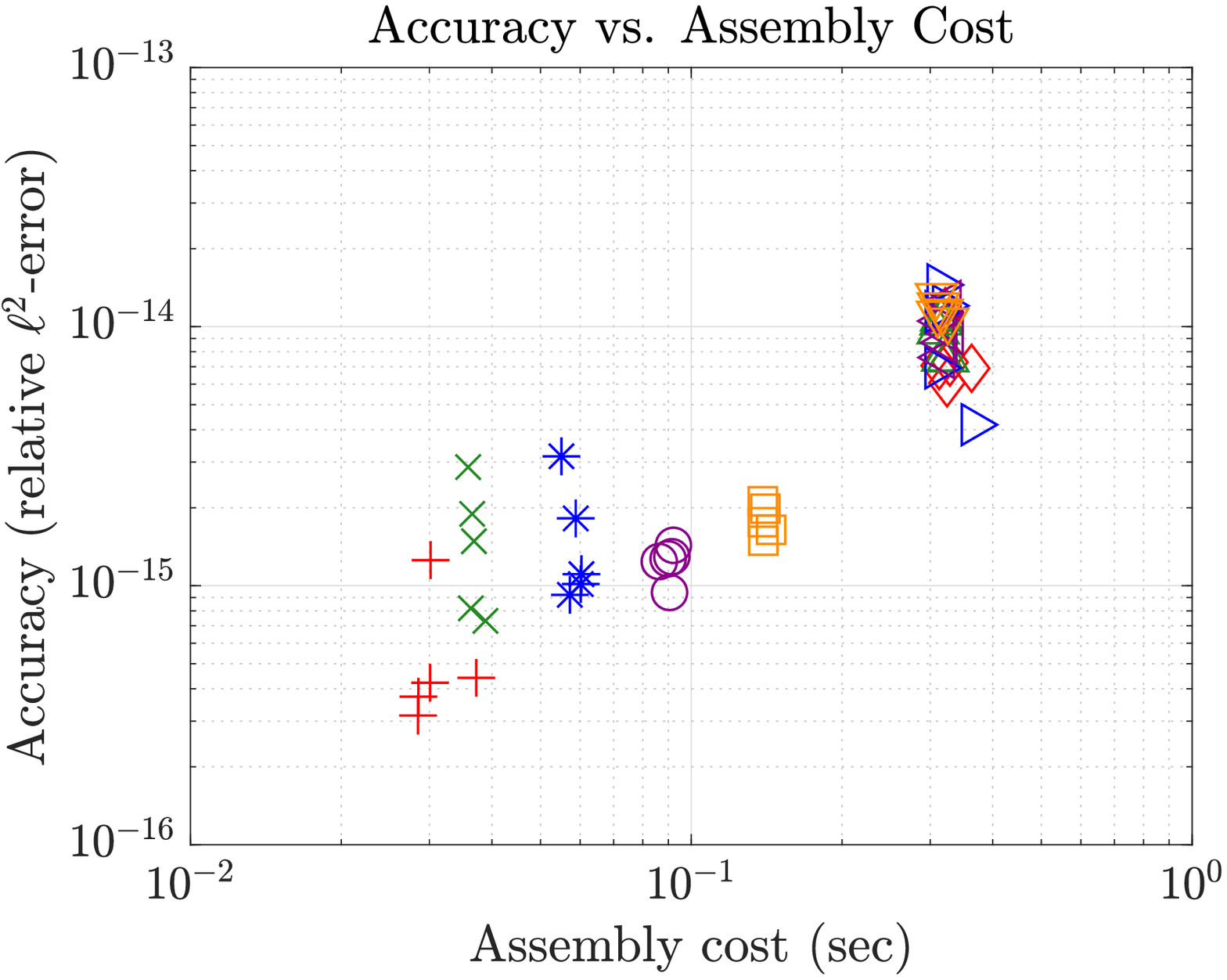}&
\raisebox{0.7cm}{\includegraphics[height = 2.5cm]{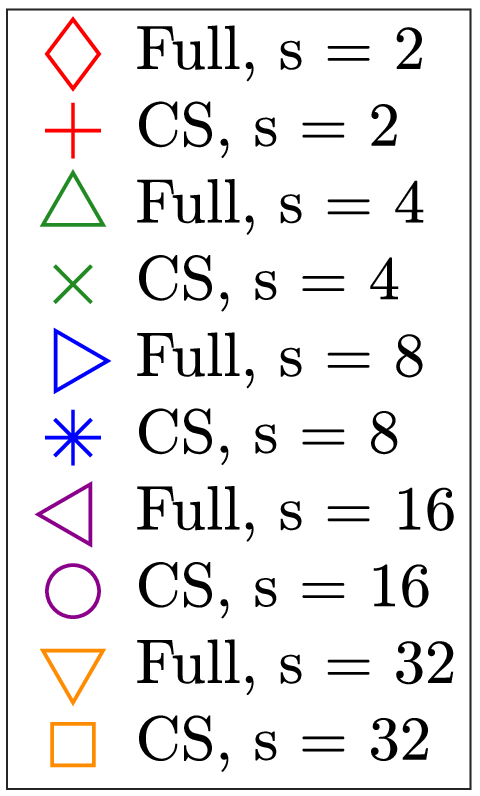}}
\end{tabular}
\caption{\label{figure1}Accuracy vs. cost plots for the full and the compressive spectral collocation approaches  for the recovery of randomly generated $s$-sparse solutions to the diffusion equation with nonconstant coefficient $\eta$ defined by \eqref{eq:nonconstant_coefficient}. Different colors refer to different sparsity levels: $s=2$ (red), $s=4$ (green), $s = 8$ (blue), $s = 16$ (magenta), and $s = 32$ (orange). The markers $\diamond$, $\triangle$, $\triangleright$, $\triangleleft$, $\triangledown$ correspond to the full approach with $s = 2,4,8,16,32$, respectively. The markers $+$,$\times$,$\ast$, $\circ$, and $\square$ refer to the compressive approach with $s = 2,4,8,16,32$, respectively.}
\end{figure}
The computational cost is distinguished in assembly and recovery cost:
\begin{itemize}
\item For the full approach, the assembly cost is the time employed to build $B$ and $c$ as in \eqref{eq:def_B_c} and the recovery cost is the time employed by the backslash \textsc{Matlab$^\text{\textregistered}$} command to solve the linear system \eqref{eq:full_SC_system}. 
\item For the compressive approach, the assembly cost is the time employed to randomly generate the multi-indices $\tau_1,\ldots,\tau_m$ and to build $A$ and $b$ as in \eqref{eq:def_A_b_CSC} and the recovery cost is the time needed to recover the solution to \eqref{eq:CSC_system} via OMP, including the time to normalize the columns of $A$ with respect to the $\ell^2$ norm and the time to rescale the entries of the OMP solution accordingly, as prescribed by Algorithm~\ref{alg:CSC}. 
\end{itemize}
Both approaches have a high level of accuracy, around $10^{-15}$ and $10^{-14}$. The markers corresponding to the compressive approach are closer to the lower left corner of the plot. This shows that when dealing with exact sparsity, the compressive approach is more advantageous both in terms of accuracy and of computational cost.

Let us now assess cost and accuracy of both approaches in a more systematic way. Fig.~\ref{figure2} shows the box plots generated after repeating the same random experiment as before 100 times and for $s = 2,4,8,16,32,64$. Recalling \eqref{eq:choice_m_K}, the corresponding numbers of collocation points are $m = 28, 56, 111, 222, 444, 888$. For the full approach we also compare backslash with OMP.
In practice, the backslash approach simply computes a solution to \eqref{eq:full_SC_system} as $B \backslash c$, whereas the OMP-based approach computes an $s$-sparse approximate solution to \eqref{eq:full_SC_system} (up to normalization of the columns of $B$) via OMP.
\begin{figure}[t]
\centering
\begin{tabular}{ccc}
\hspace{0.8cm} Full (backslash recovery) & Full (OMP recovery) & Compressive \\
\includegraphics[height = 3.5cm]{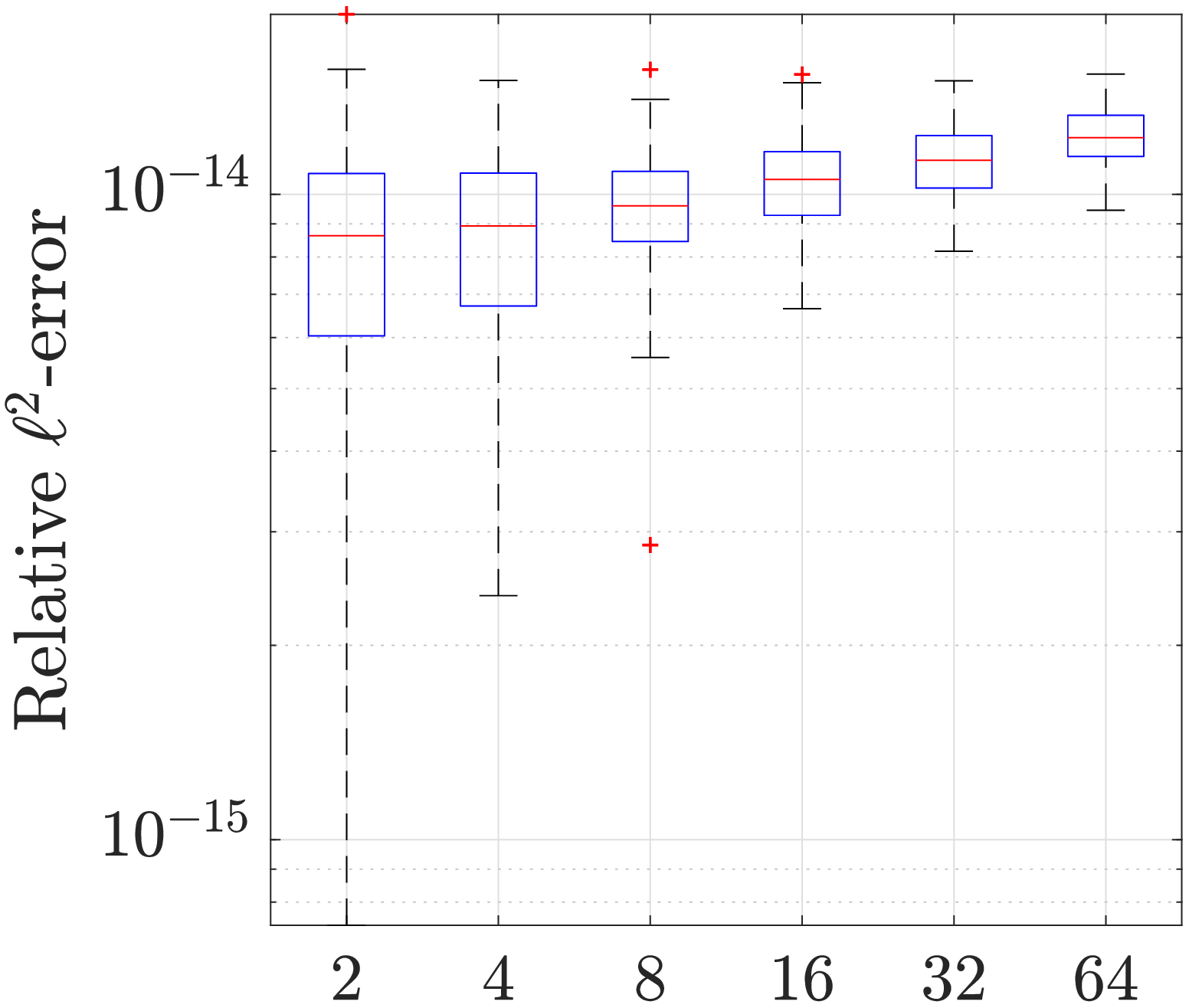} &
\includegraphics[height = 3.5cm]{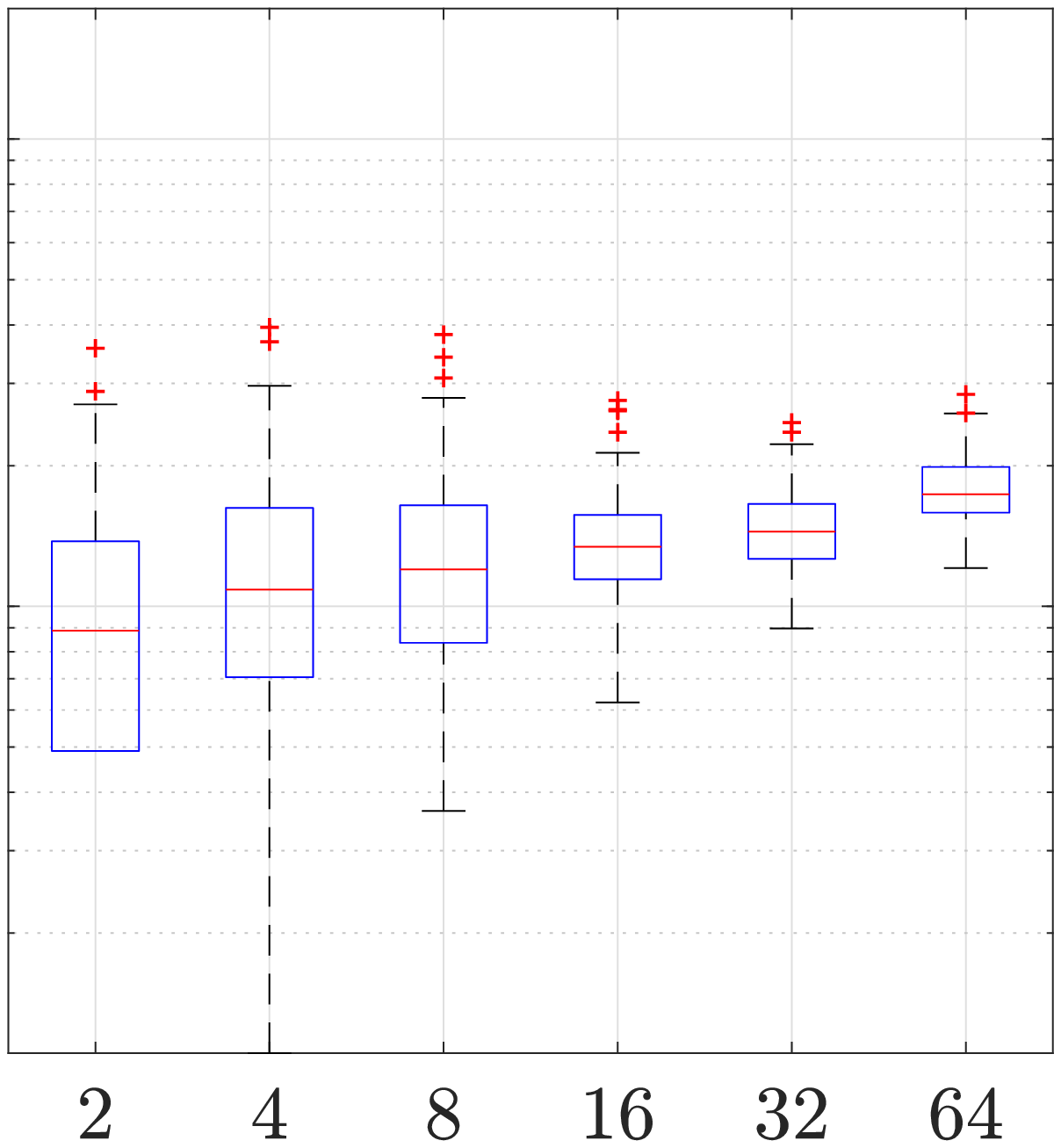} &
\includegraphics[height = 3.5cm]{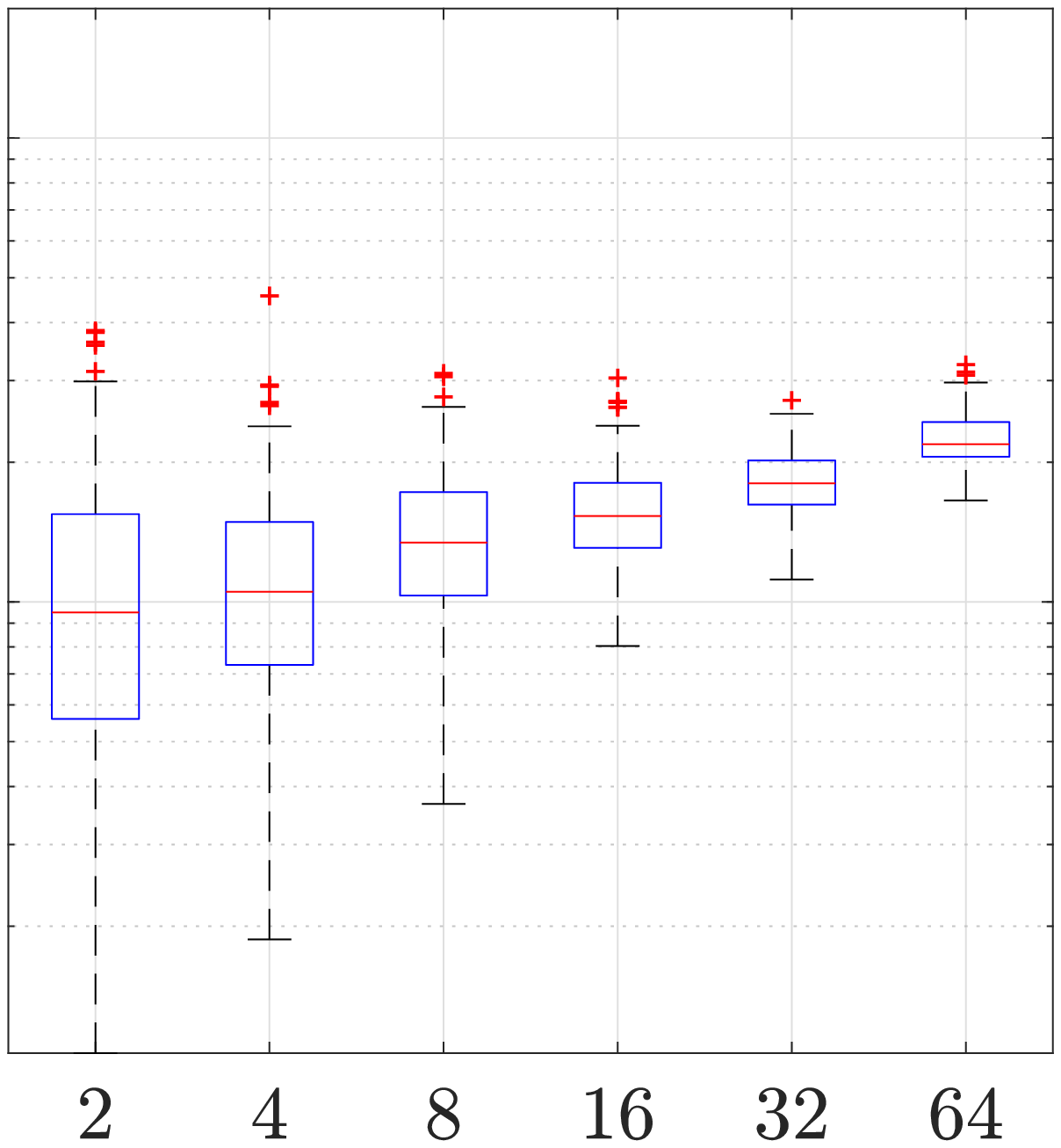} \\
\includegraphics[height = 3.5cm]{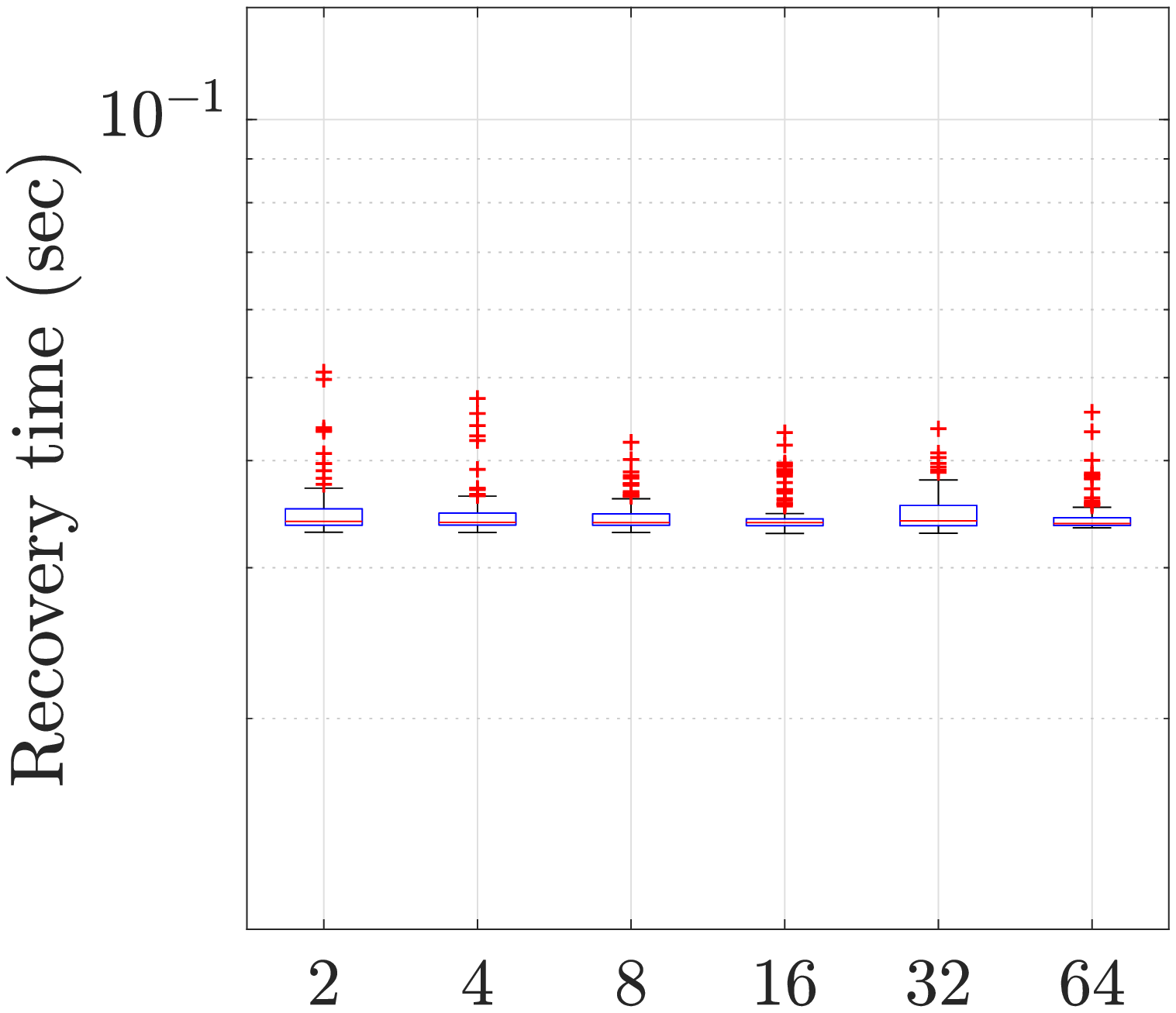} &
\includegraphics[height = 3.5cm]{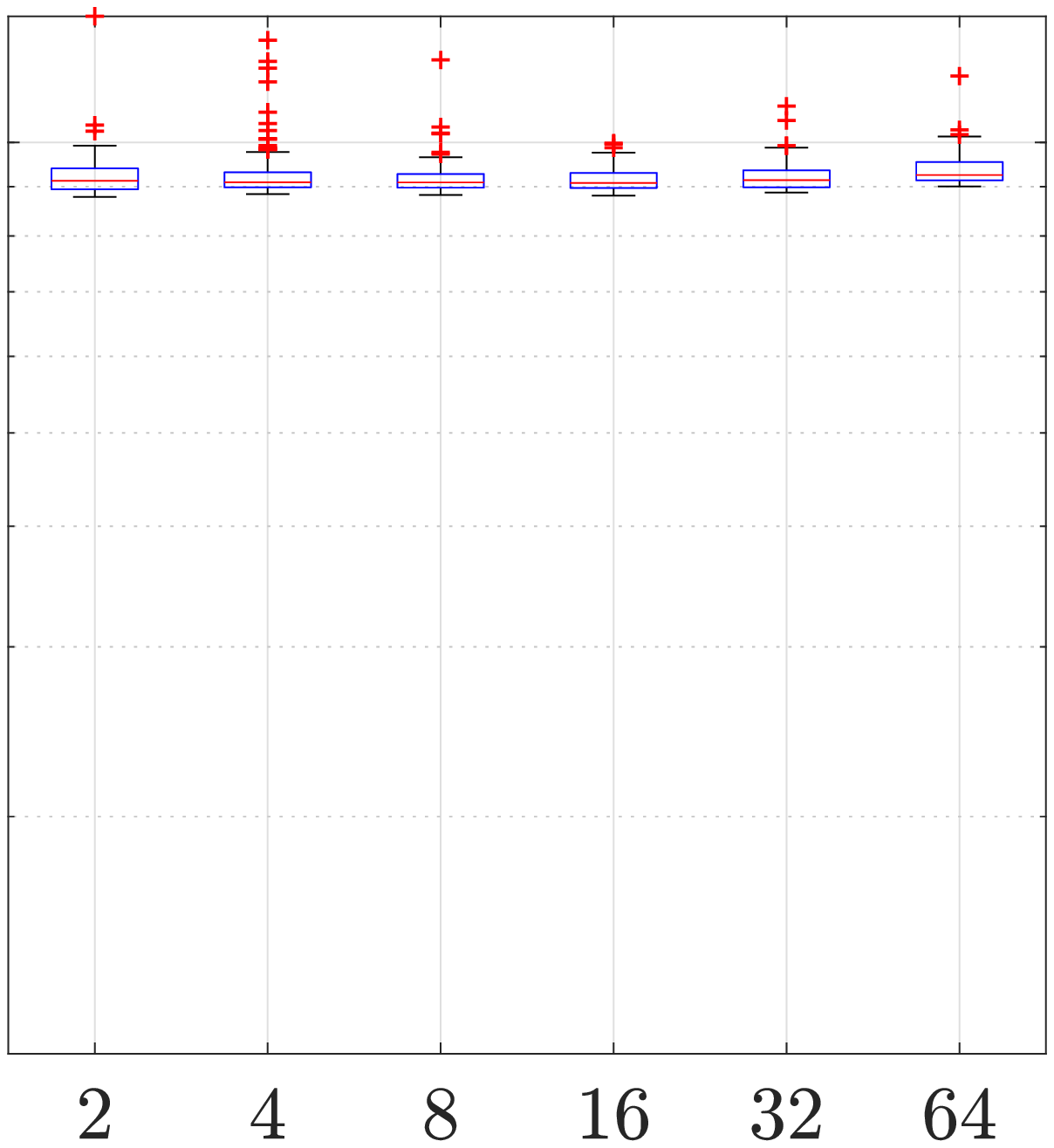} &
\includegraphics[height = 3.5cm]{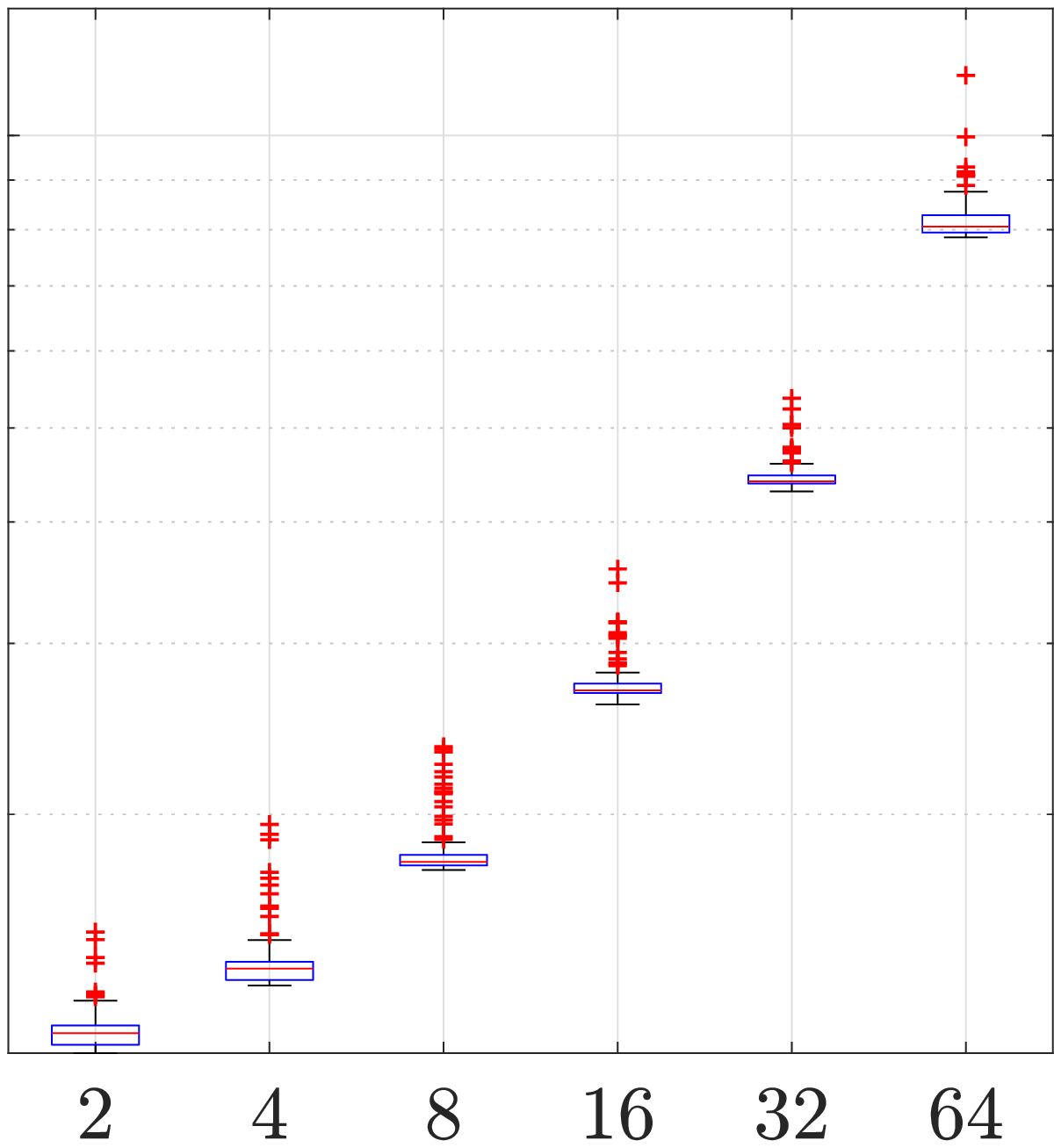} \\
\includegraphics[height = 3.8cm]{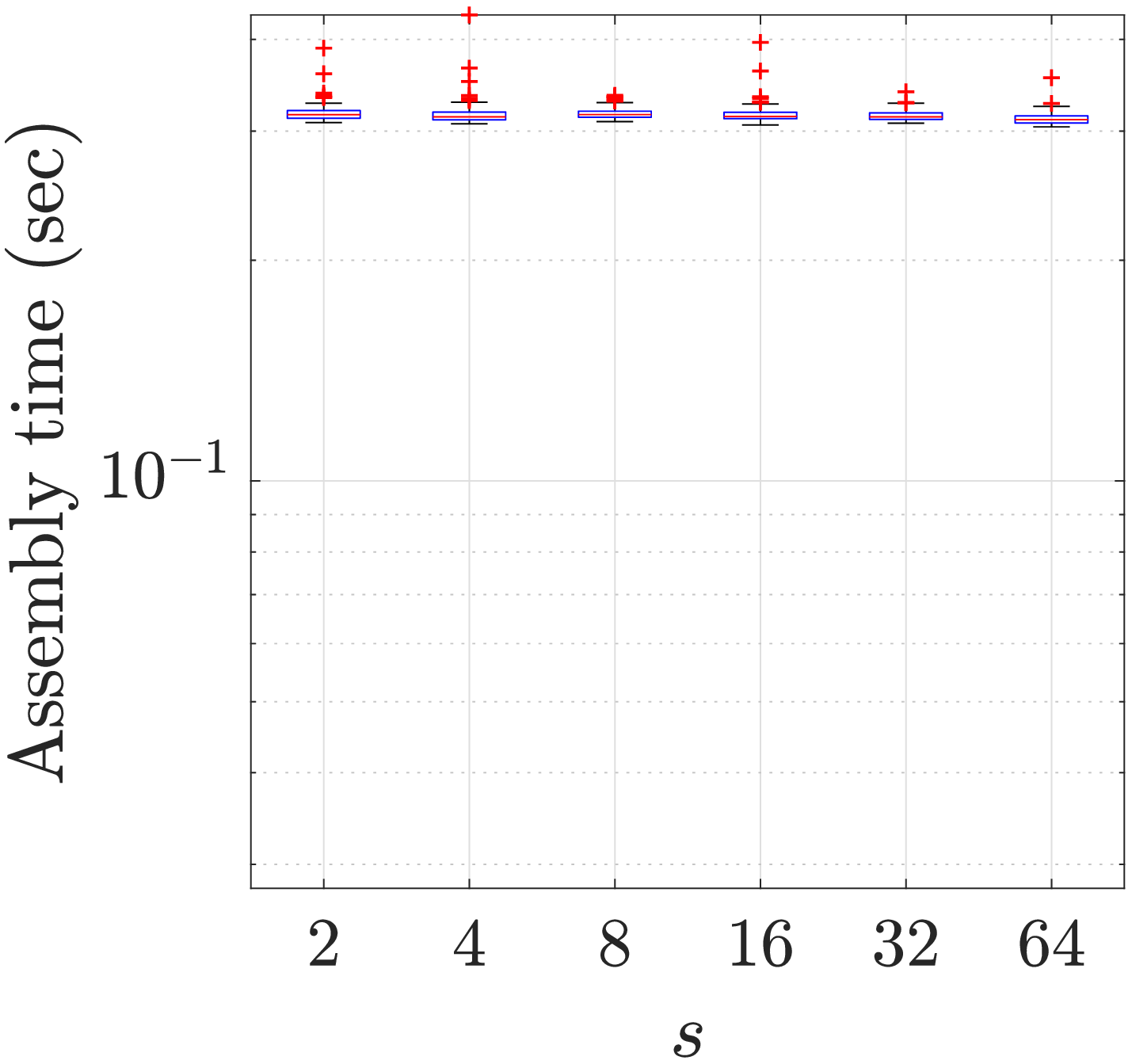} & 
\includegraphics[height = 3.8cm]{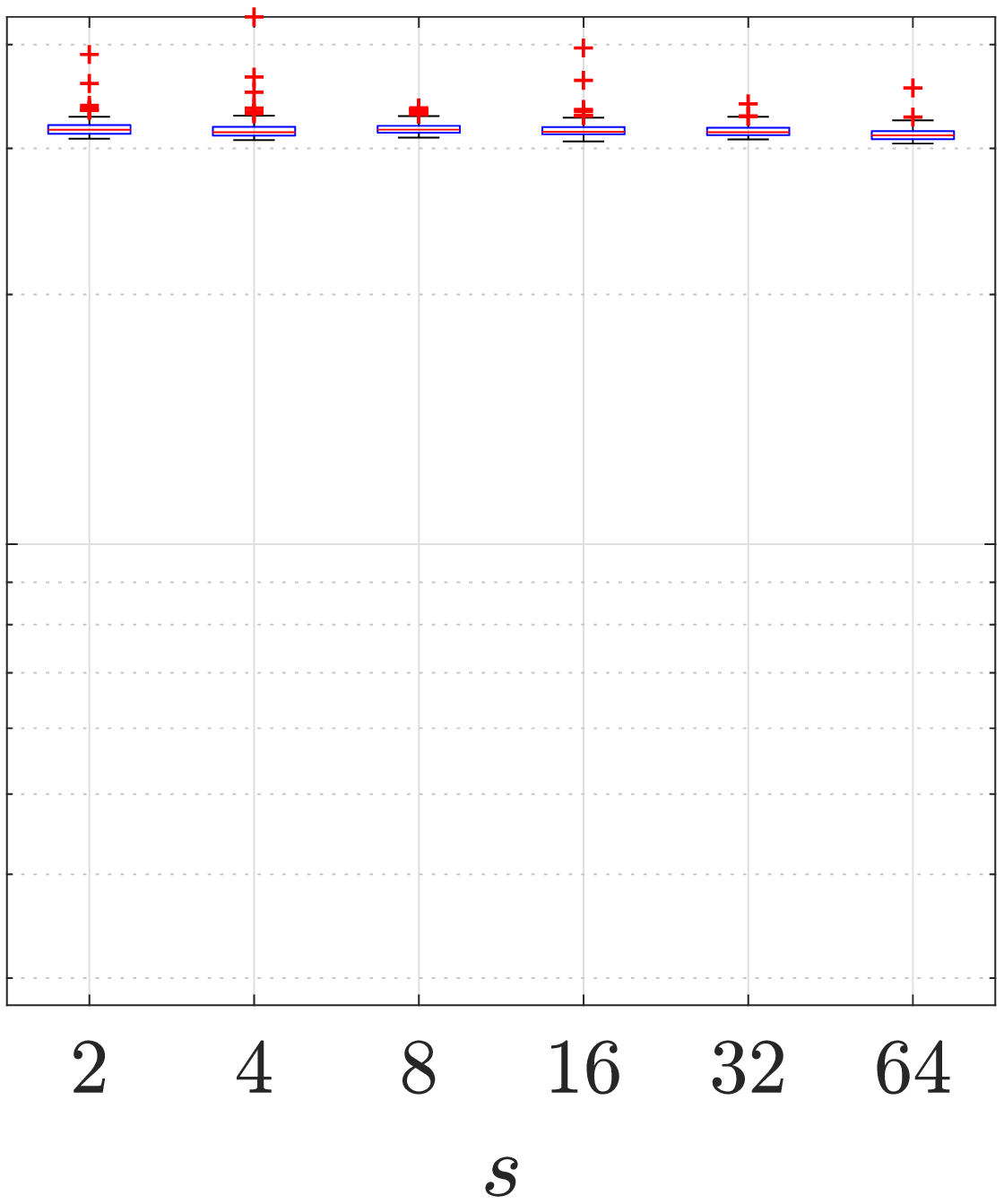} &
\includegraphics[height = 3.8cm]{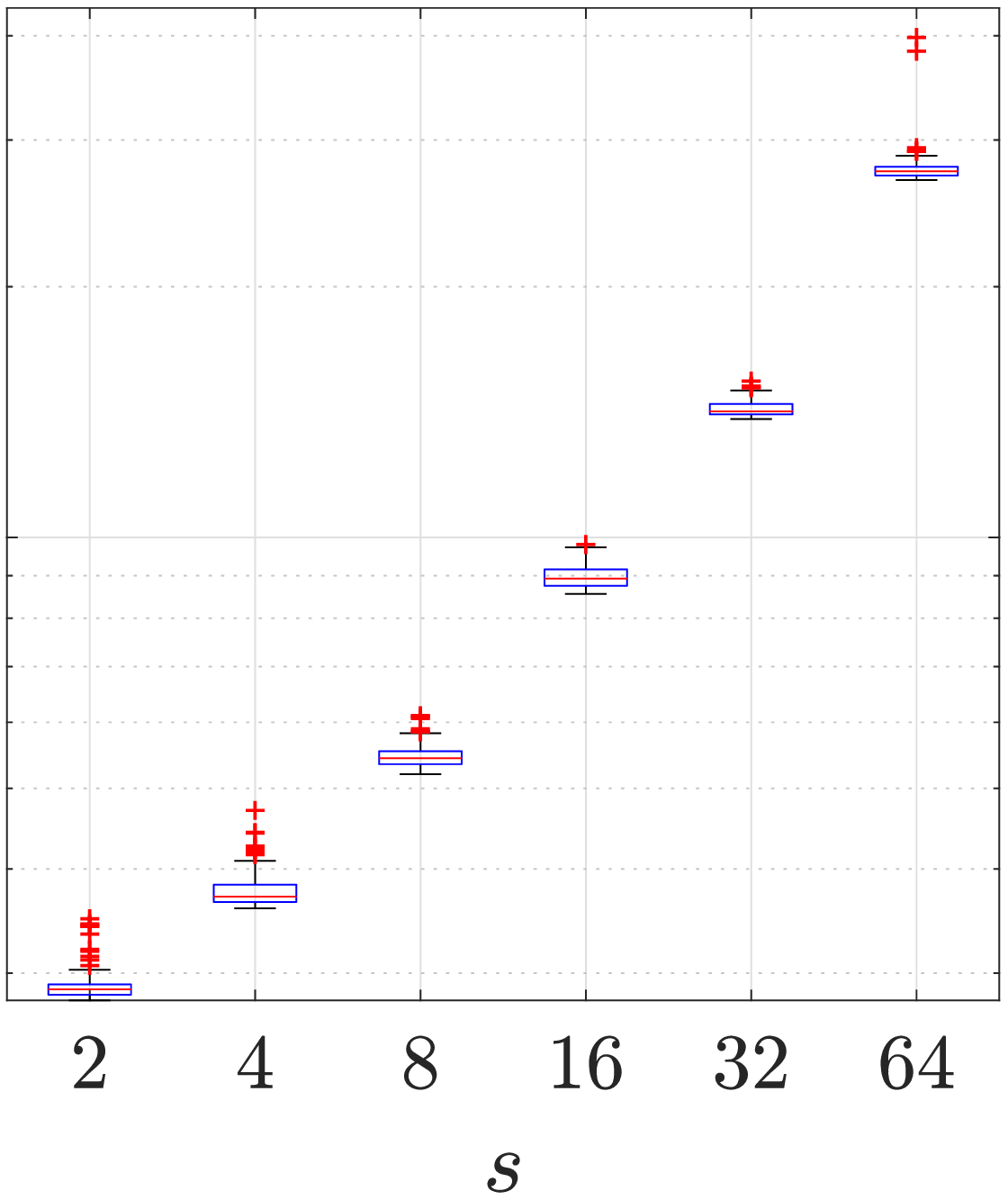} 
\end{tabular}
\caption{\label{figure2}Performance analysis of full and compressive spectral collocation from the accuracy and computational cost viewpoints for the recovery of randomly generated $s$-sparse solutions to the diffusion equation with nonconstant coefficient $\eta$ defined by \eqref{eq:nonconstant_coefficient}. The box plots are referred to 100 random runs.}
\end{figure}

The very good level of accuracy of both approaches is confirmed by this second experiment. It is remarkable that the recovery error of the compressive approach is slightly better than that of the full approach, especially for smaller sparsities. We  observe that, in general, the compressive approach outperforms the full one both in terms of accuracy \emph{and} computational cost. In general, the smaller the sparsity $s$, the higher the computational cost reduction gained by compressing the discretization. By looking at the second column, we can see that, in the full case, OMP is able to compute more accurate solutions with respect to the backslash. This is arguably due to the fact that the largest least-squares problem solved by OMP (during the $s$-th iteration) is associated with an $N \times s$ submatrix of the full $N \times N$ discretization matrix linear system. Therefore, the former matrix is, in general, better conditioned than the latter.
\footnote{The substantial independence of the OMP recovery cost with respect to $s$ for the full approach depends on two factors: the particular implementation of OMP in the package \textsf{OMP-Box} and the normalization step $\widetilde{A} = A M^{-1}$ in Algorithm~3.1. 
In fact, in order to speed up the OMP iteration, the function \texttt{omp} of \textsf{OMP-Box} used to produce these results takes $\widetilde{A}^T\widetilde{A}$ as input. When $A$ is $N \times N$, the cost of computing the matrices $\widetilde{A}$ and $\widetilde{A}^T\widetilde{A}$ is independent of $s$ and it turns out to be consistently larger than the cost of OMP itself. As a result, the effect of $s$ on the overall computational cost is negligible. The same remark holds for Fig.~4.}

\subsection{Recovery of compressible solutions}

We compare the full and the compressive approaches for the recovery of compressible solutions. We will test the methods for the recovery of the the exact solution
\begin{equation}
\label{eq:exact_solution}
u(z) = (16 \, z_1 \,  z_2 \, (1-z_1) (1-z_2))^2, \quad \forall z \in \overline{\Omega},
\end{equation}  
whose plot is shown in Fig.~\ref{figure3} (top left). The forcing term $F$ in \eqref{eq:diffusion_strong} is defined in order to have \eqref{eq:exact_solution} as exact solution.

Let us fix $n = 32$, corresponding to $N = 1024$, and $s = 32$. With this choice, and recalling \eqref{eq:choice_m_K}, we have $m  =444$.  Fig.~\ref{figure3} shows the results of the full and the compressive spectral collocation approaches. 
\begin{figure}[t]
\centering
\begin{tabular}{cc}
\includegraphics[height = 4.5cm]{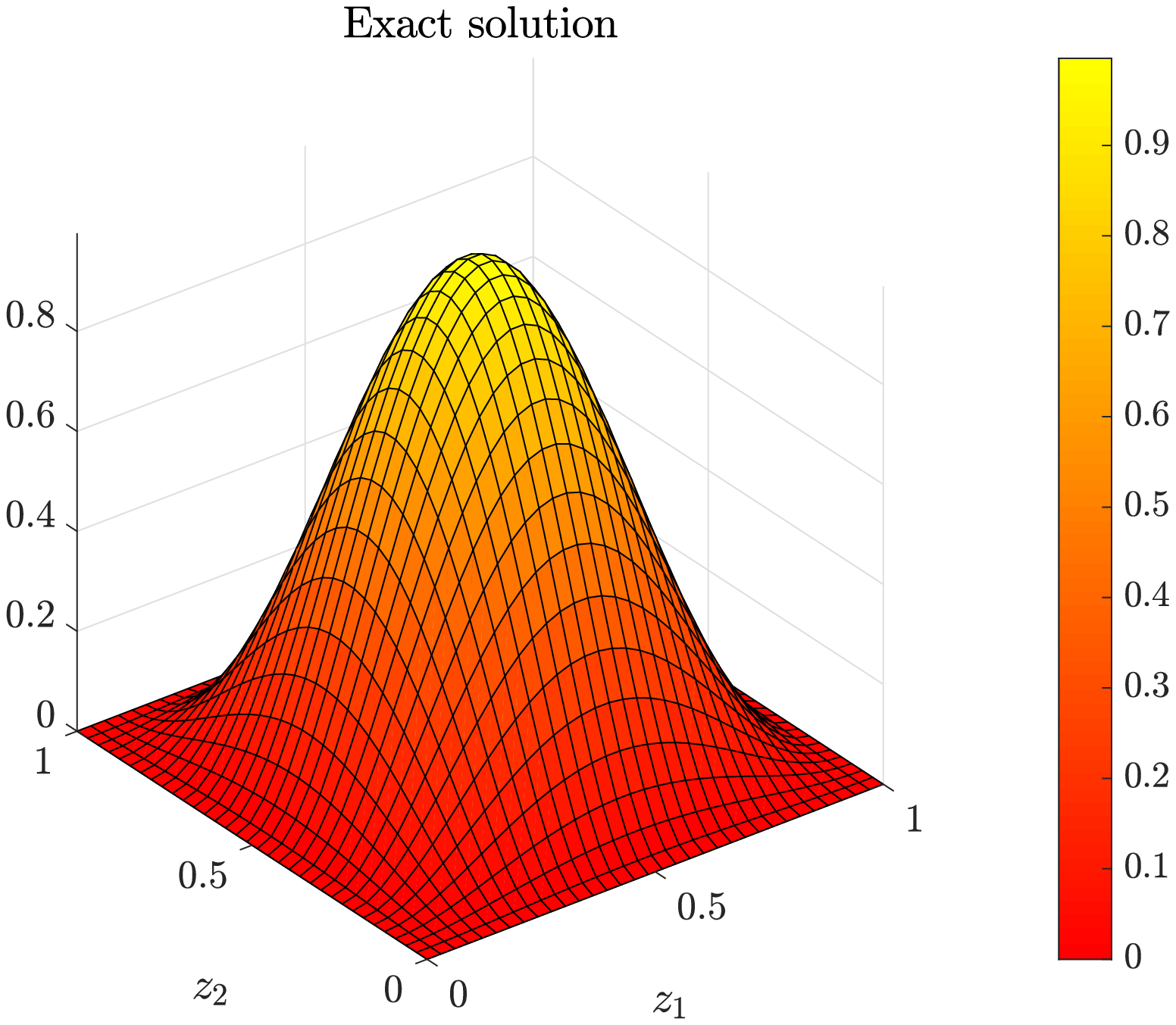} & 
\includegraphics[height = 4.5cm]{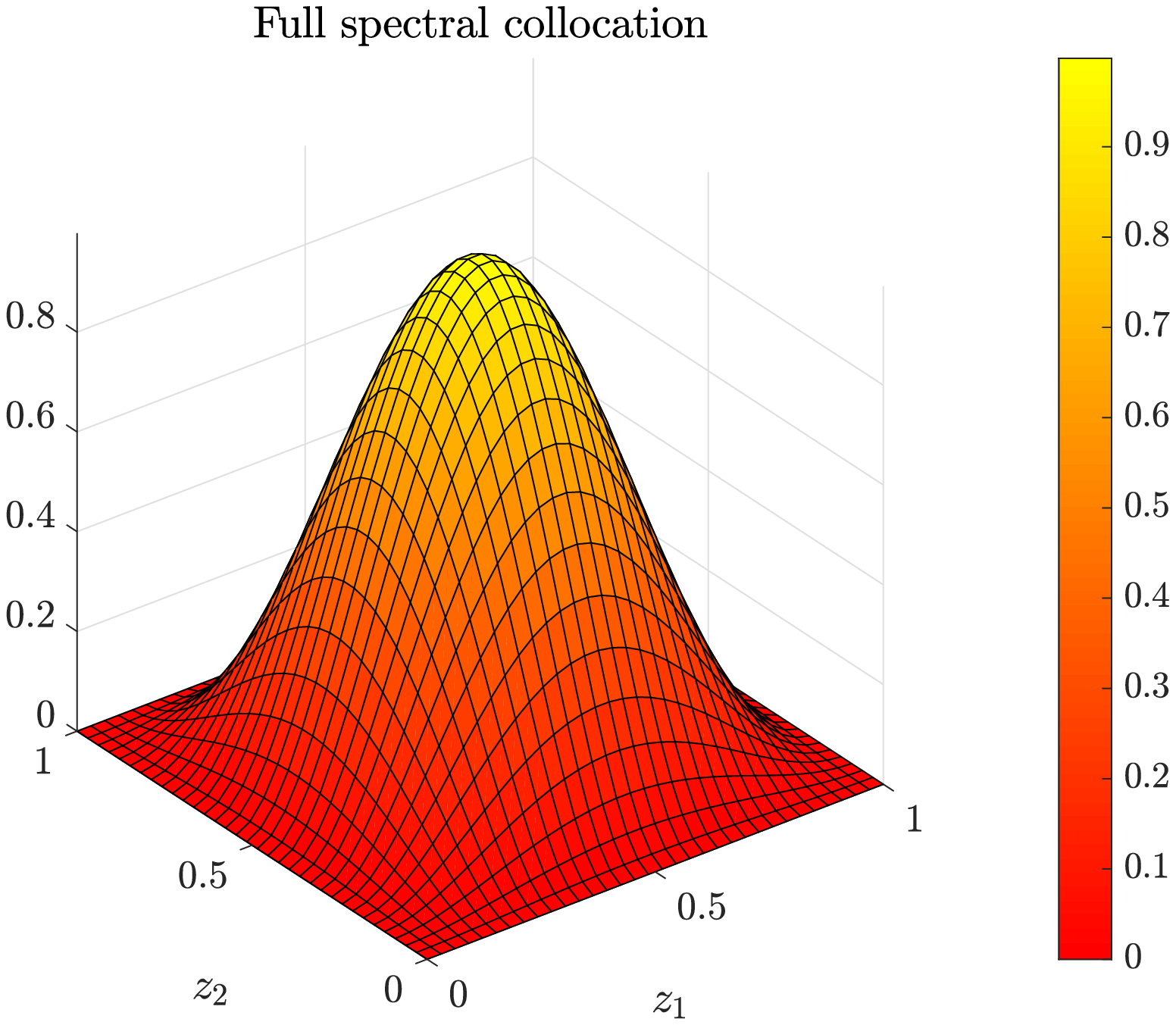} \\\includegraphics[height = 4.5cm]{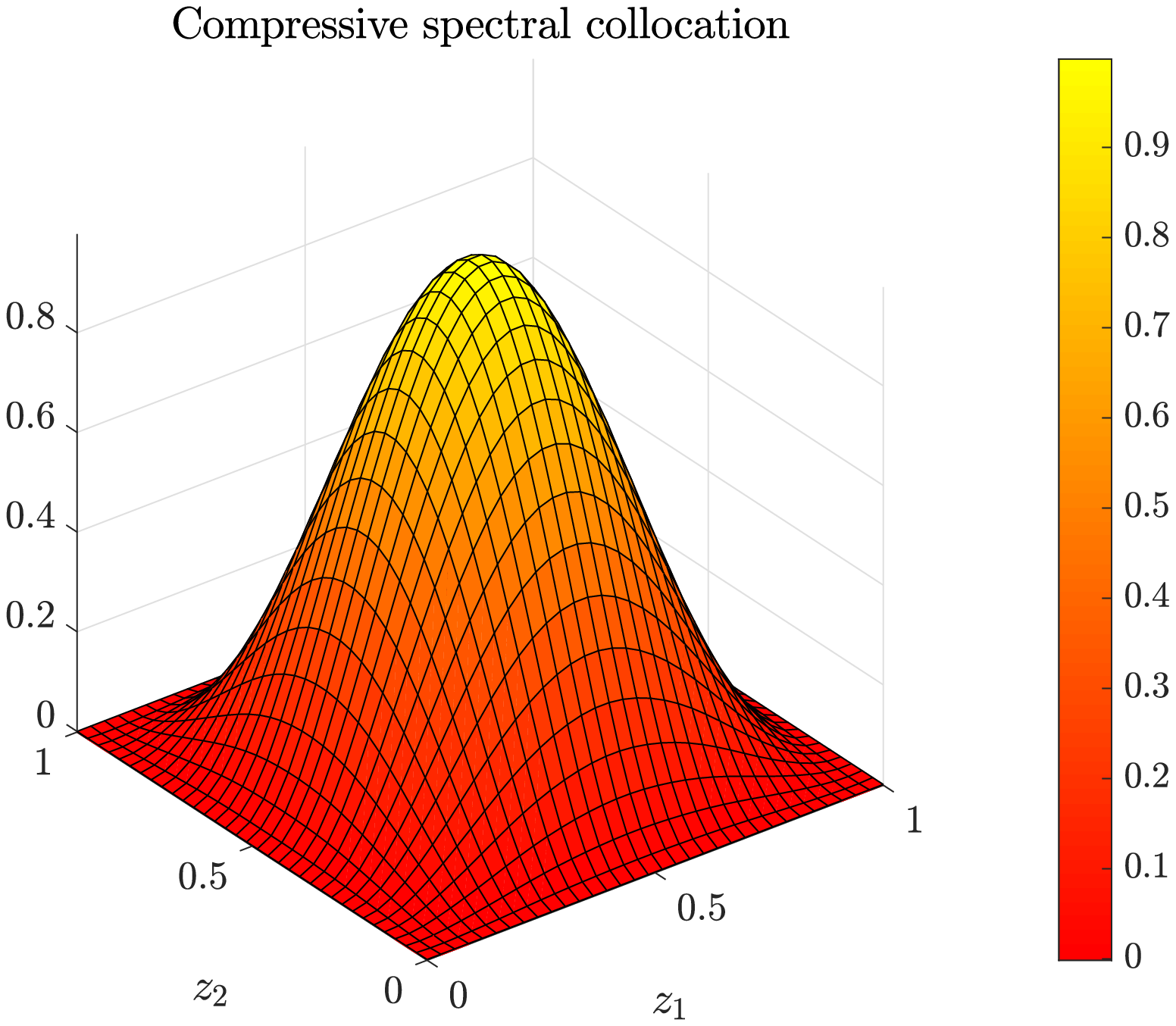} &\includegraphics[height = 4.5cm]{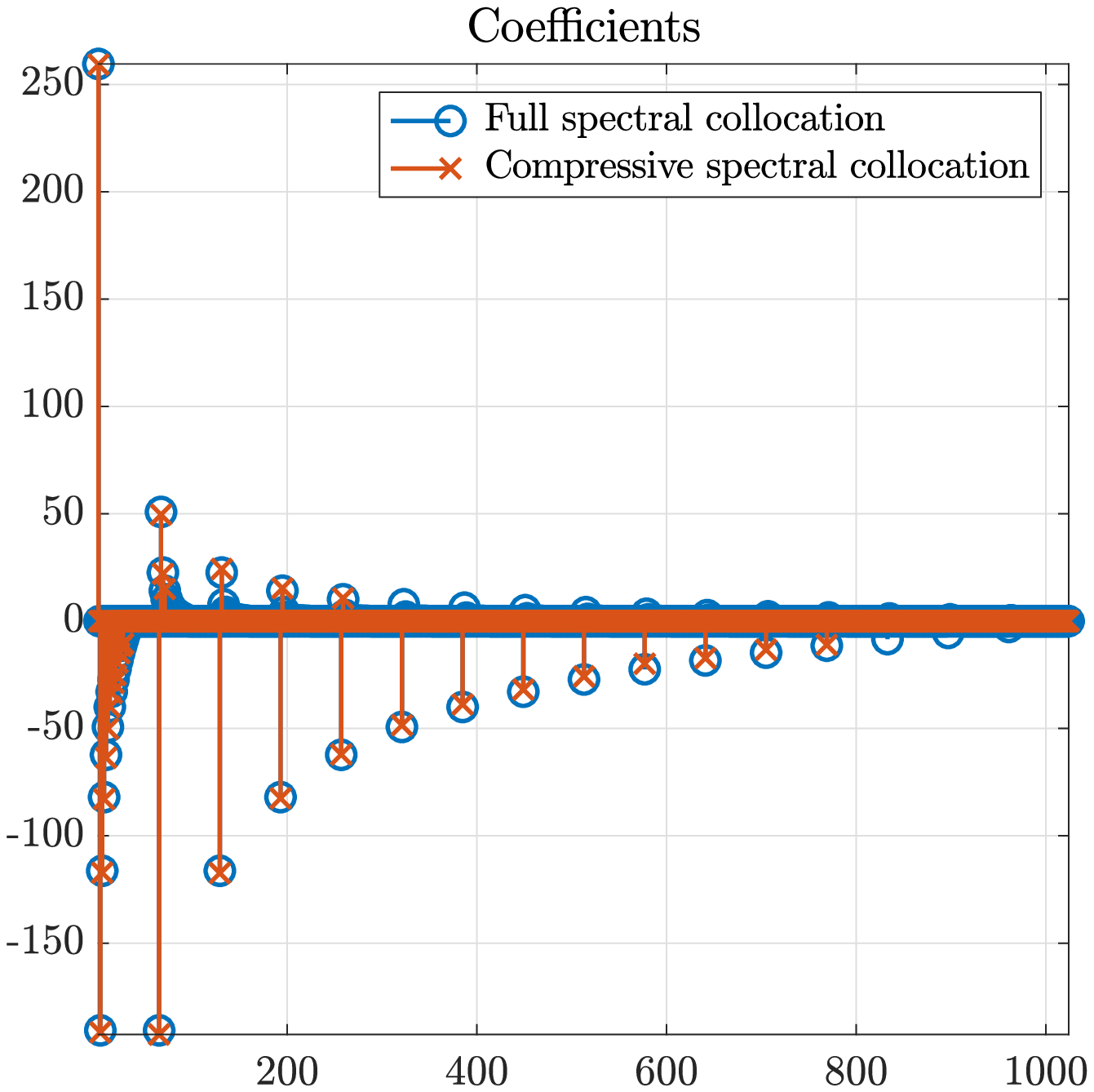}
\end{tabular}
\caption{\label{figure3}Full and compressive spectral approximation of the compressible solution \eqref{eq:exact_solution} to a diffusion equation with coefficient $\eta$ defined by \eqref{eq:nonconstant_coefficient}. Top left: exact solution defined as in \eqref{eq:exact_solution}. Top right: full spectral collocation approximation with $n = 32$. Bottom left: Compressive spectral collocation approximation with $n = 32$ and $s = 32$. Bottom right: plot of the coefficients $x^{\text{full}}$ and $\hat{x}$, corresponding to the full and compressive approximations.}
\end{figure}
Both methods produce a very good approximation to the exact solution. We can appreciate the ability of OMP to recover the largest absolute coefficients of the vector $x^{\text{full}}$ in Fig.~\ref{figure3} (bottom right). Comparing Fig.~\ref{figure3} (top right) and Fig.~\ref{figure3} (bottom left), we see that computing a 32-sparse approximation to the 1024-dimensional vector $x^{\text{full}}$ is sufficient to recover a compressive approximation that is visually indistinguishable from the full approximation, thanks to the compressibility of the solution.

In the same setting as before, we consider sparsity levels $s = 2,4,8,16,32,64$ and carry out a more extensive numerical assessment, in the same spirit as Fig.~\ref{figure2}. We repeat the previous experiment 100 times and show the corresponding box plots  in Fig.~\ref{figure4}.
\begin{figure}[t]
\centering
\begin{tabular}{ccc}
\hspace{0.8cm}Full (backslash recovery) & Full (OMP recovery) & CS  \\
\includegraphics[height = 3.6cm]{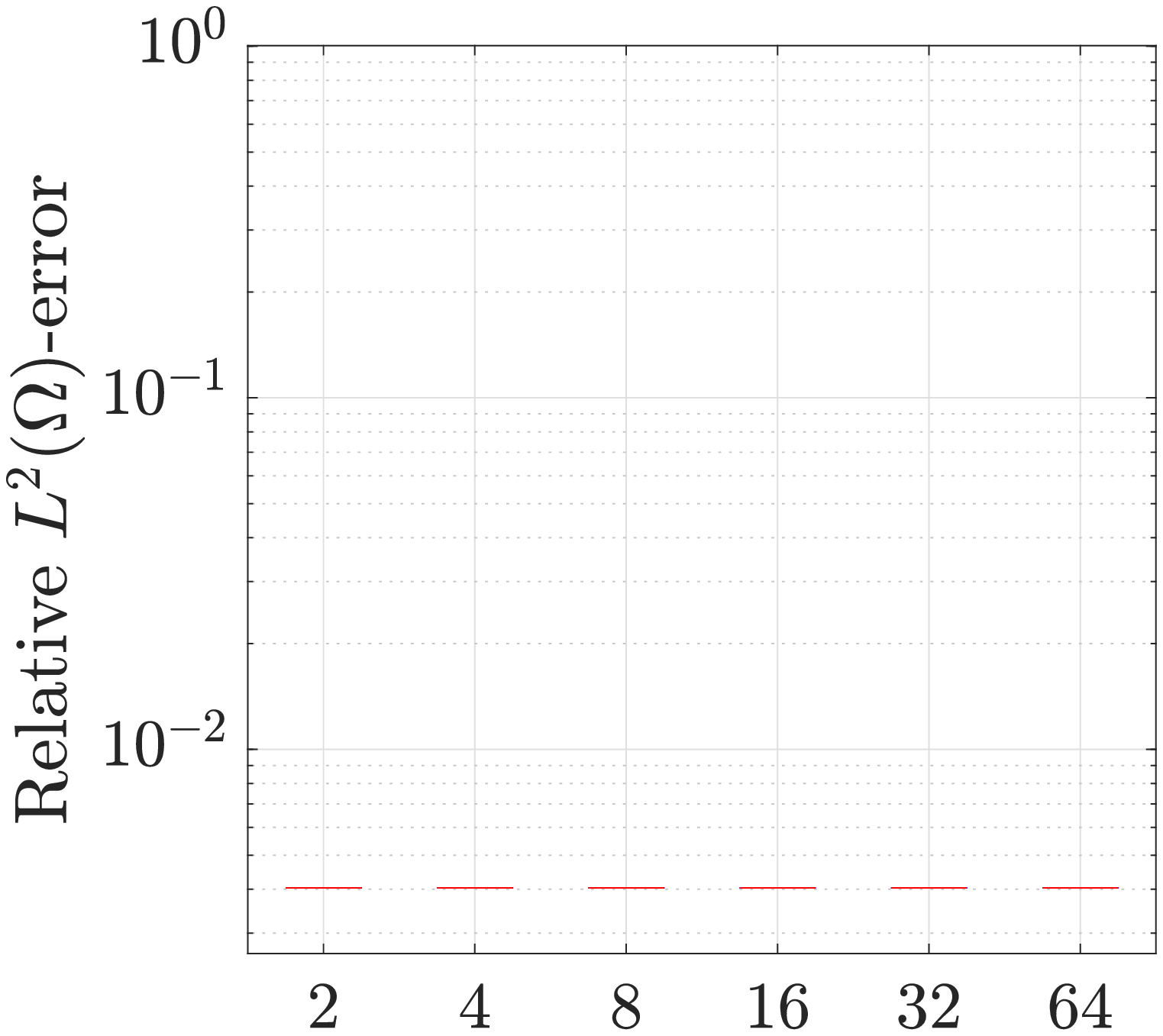} &
\includegraphics[height = 3.5cm]{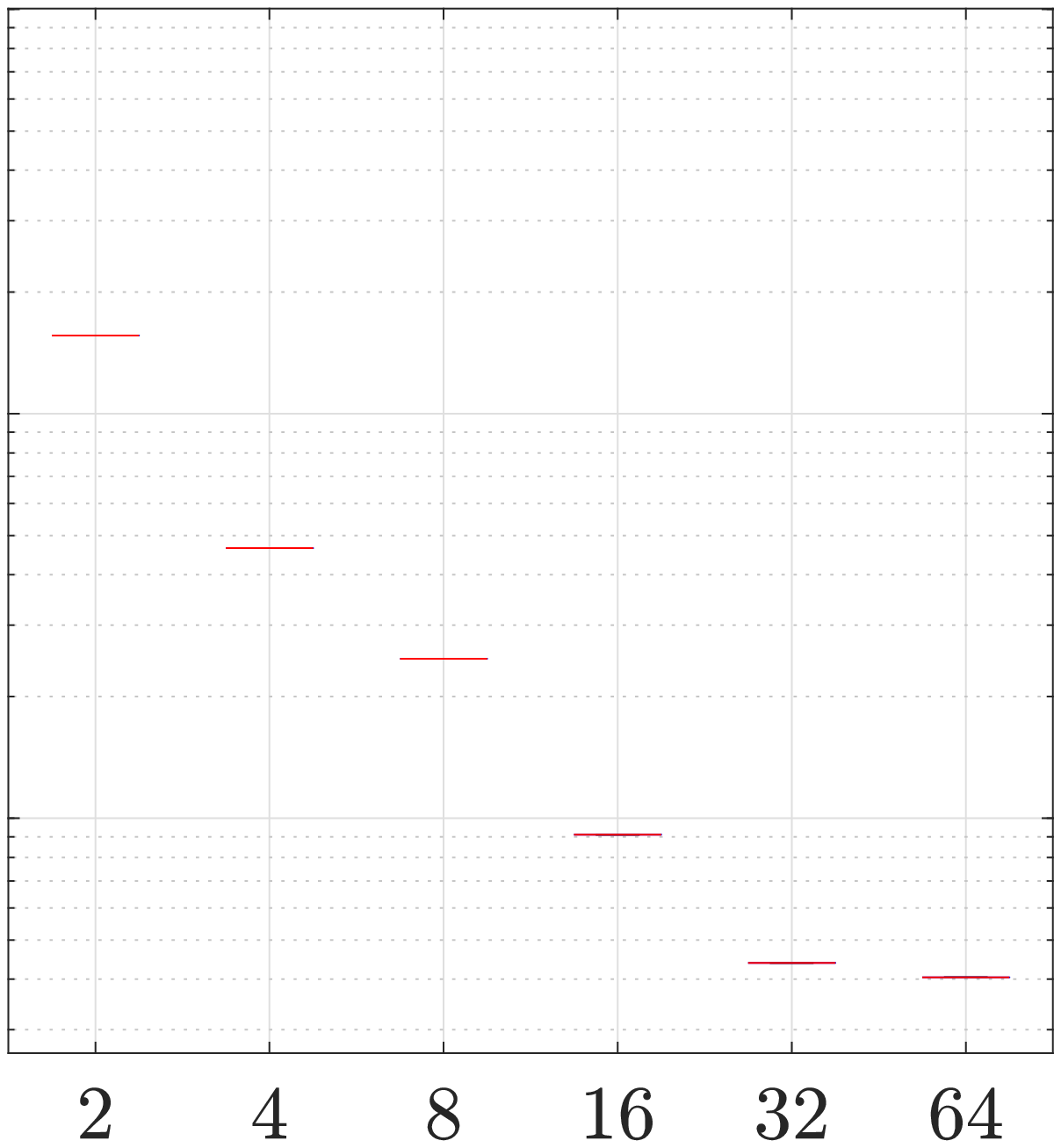} &
\includegraphics[height = 3.5cm]{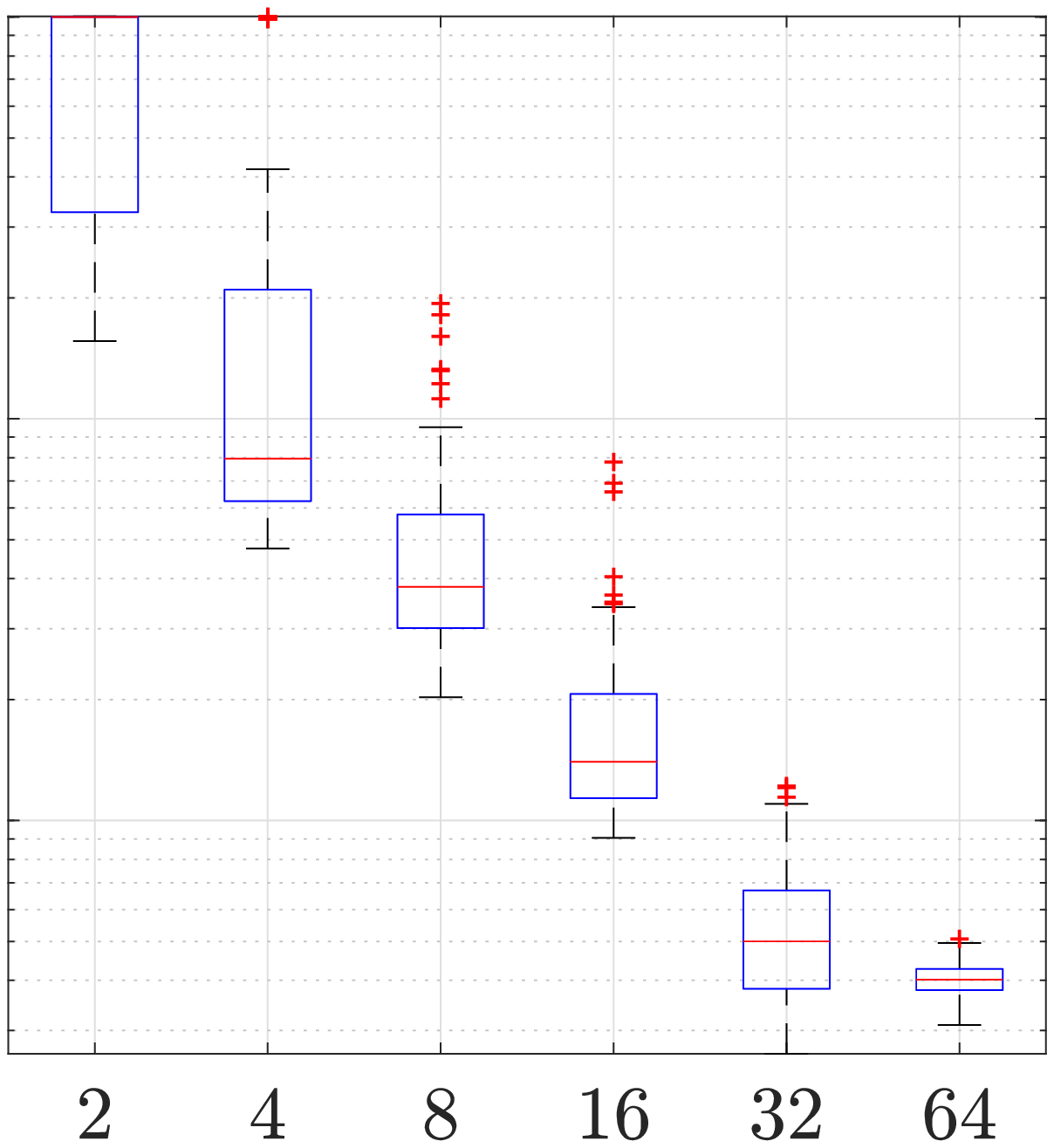} \\
\includegraphics[height = 3.5cm]{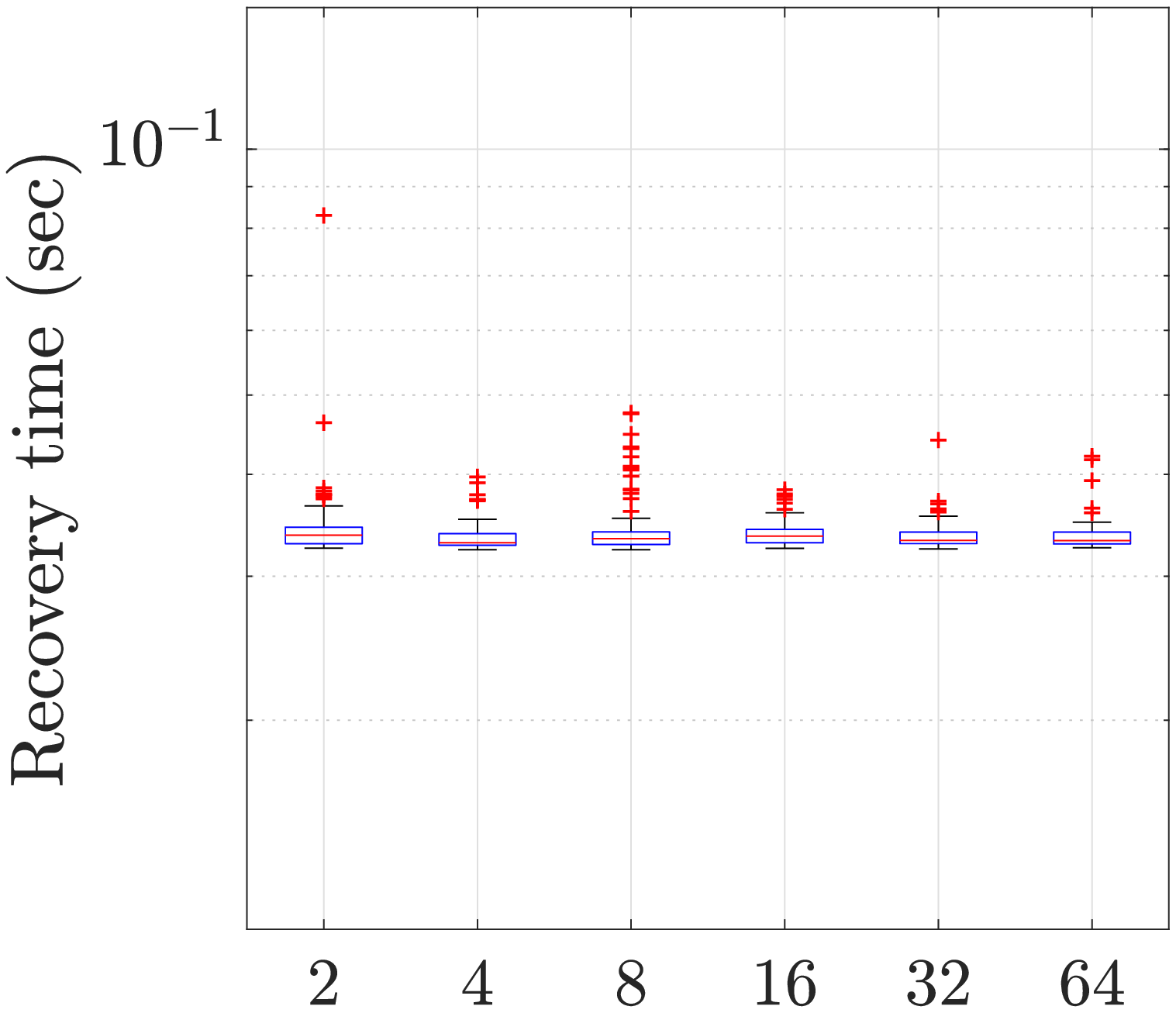} &
\includegraphics[height = 3.5cm]{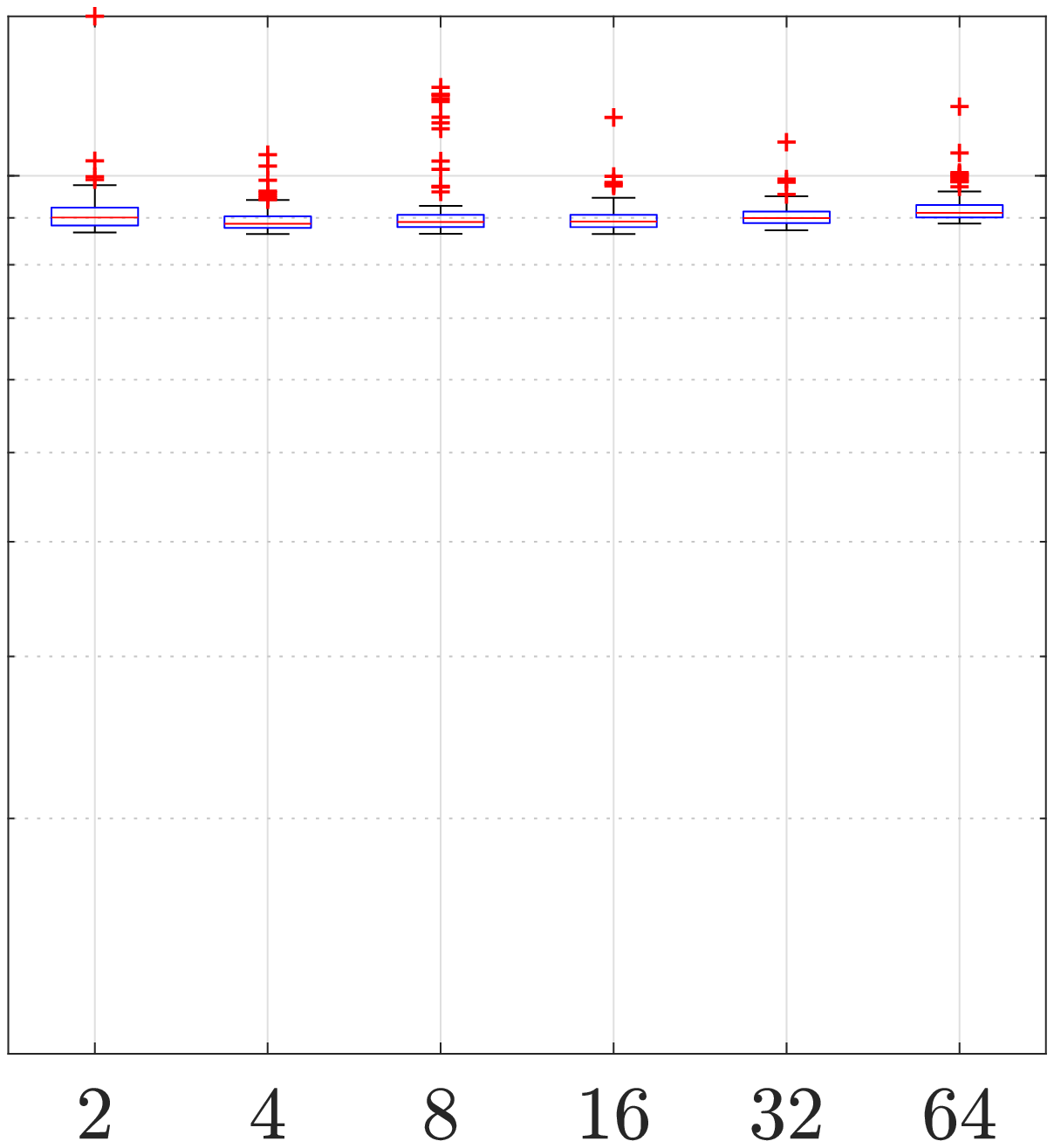} &
\includegraphics[height = 3.5cm]{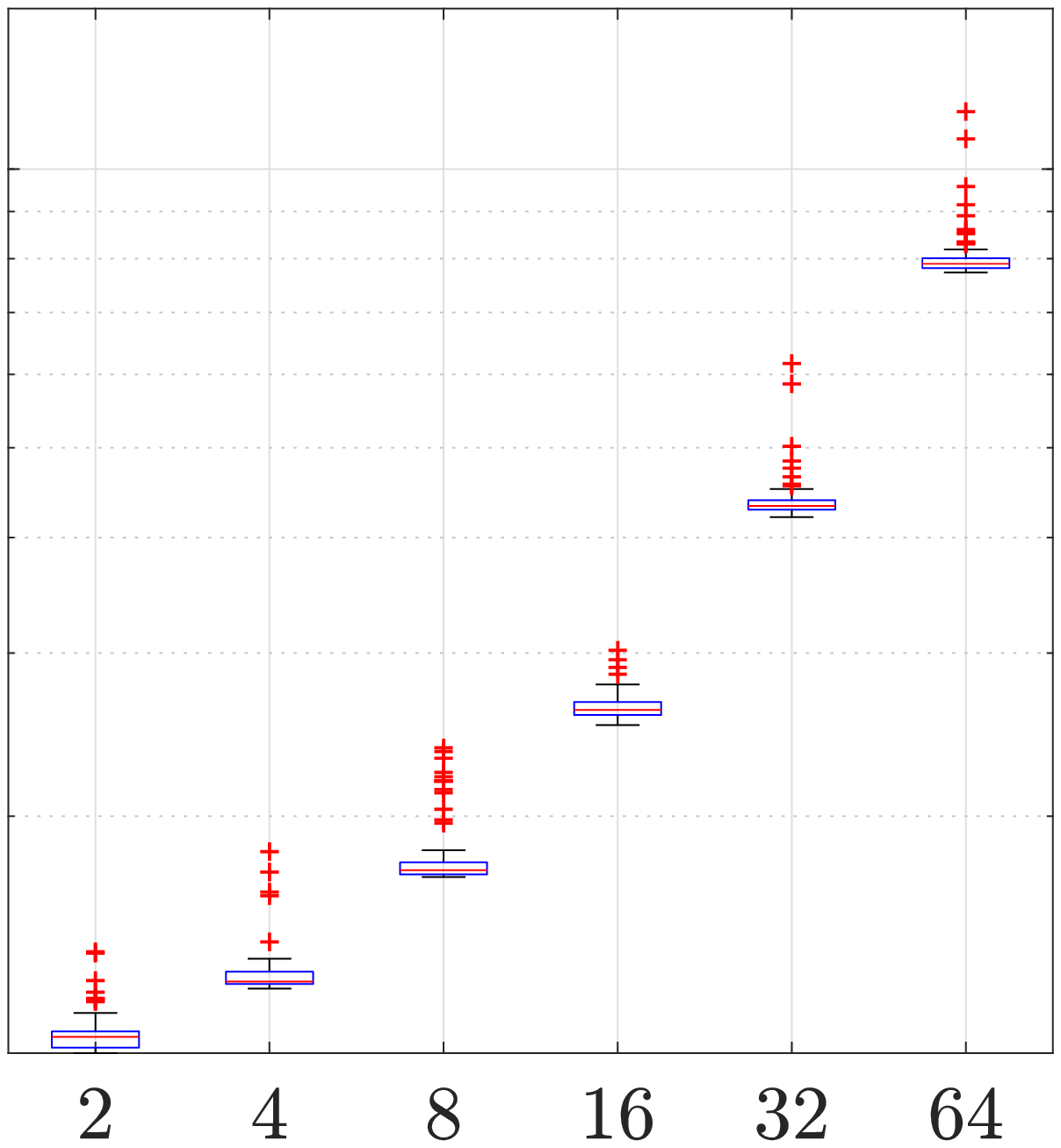} \\
\includegraphics[height = 3.8cm]{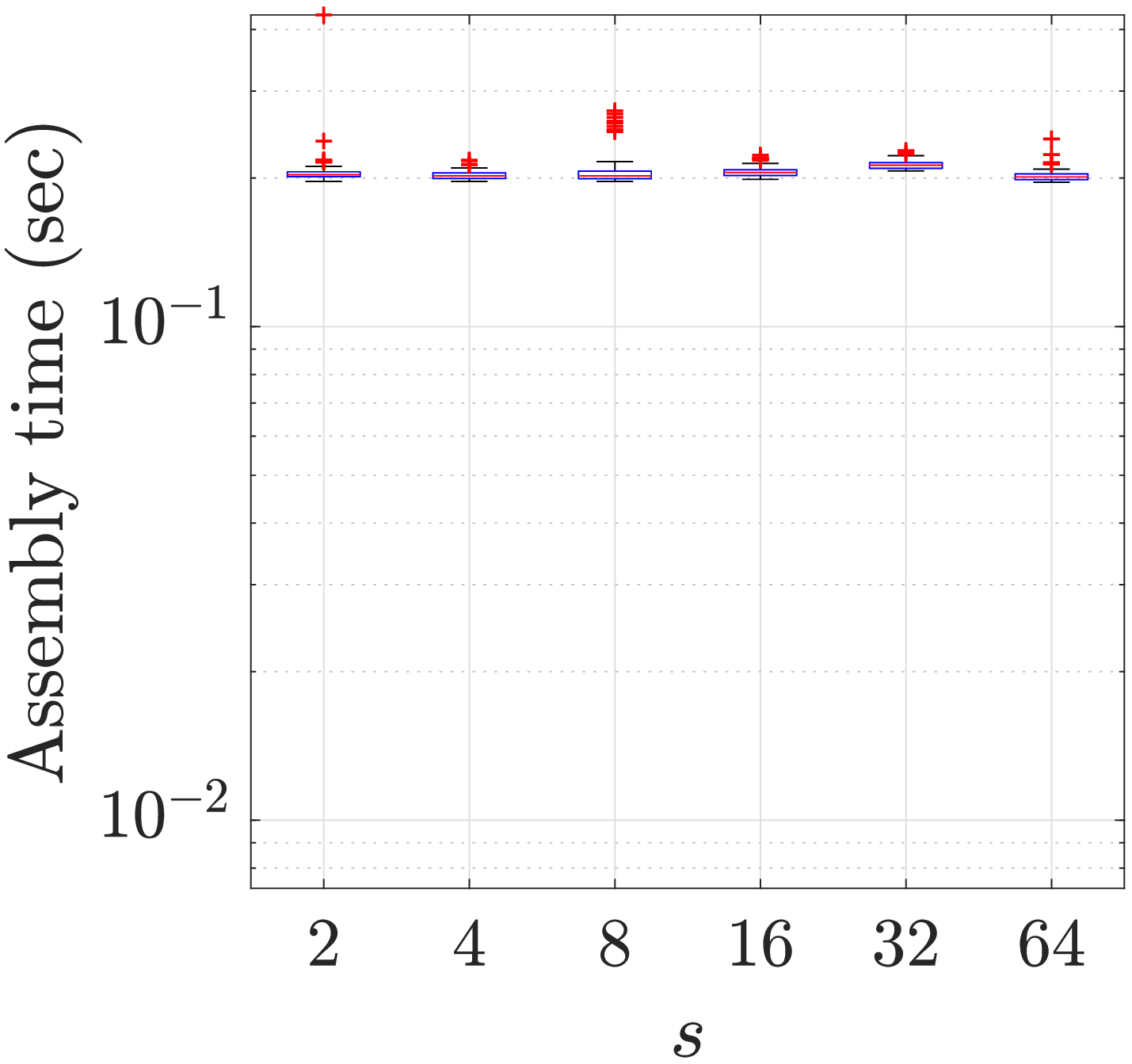} & \includegraphics[height = 3.8cm]{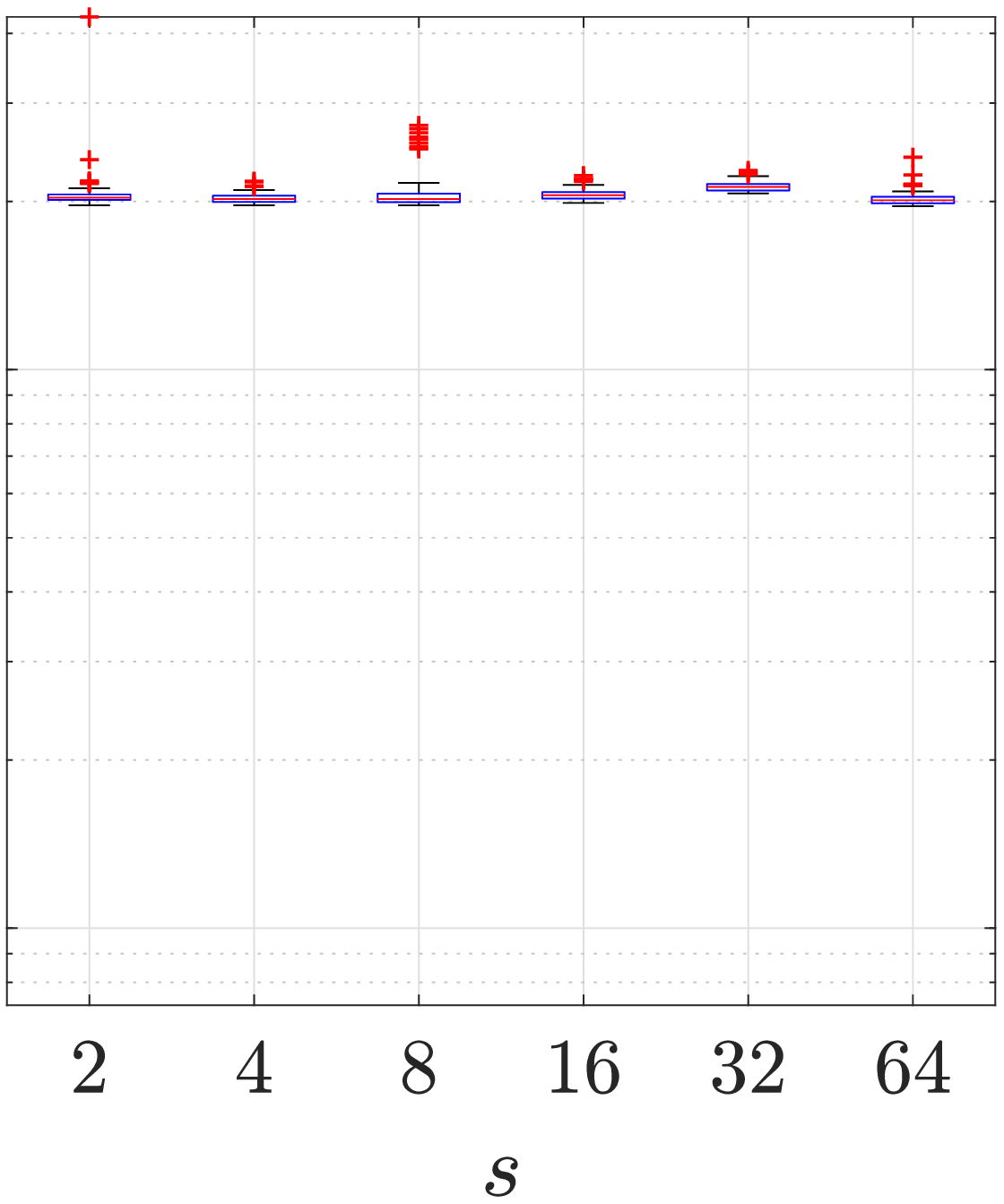} &
\includegraphics[height = 3.8cm]{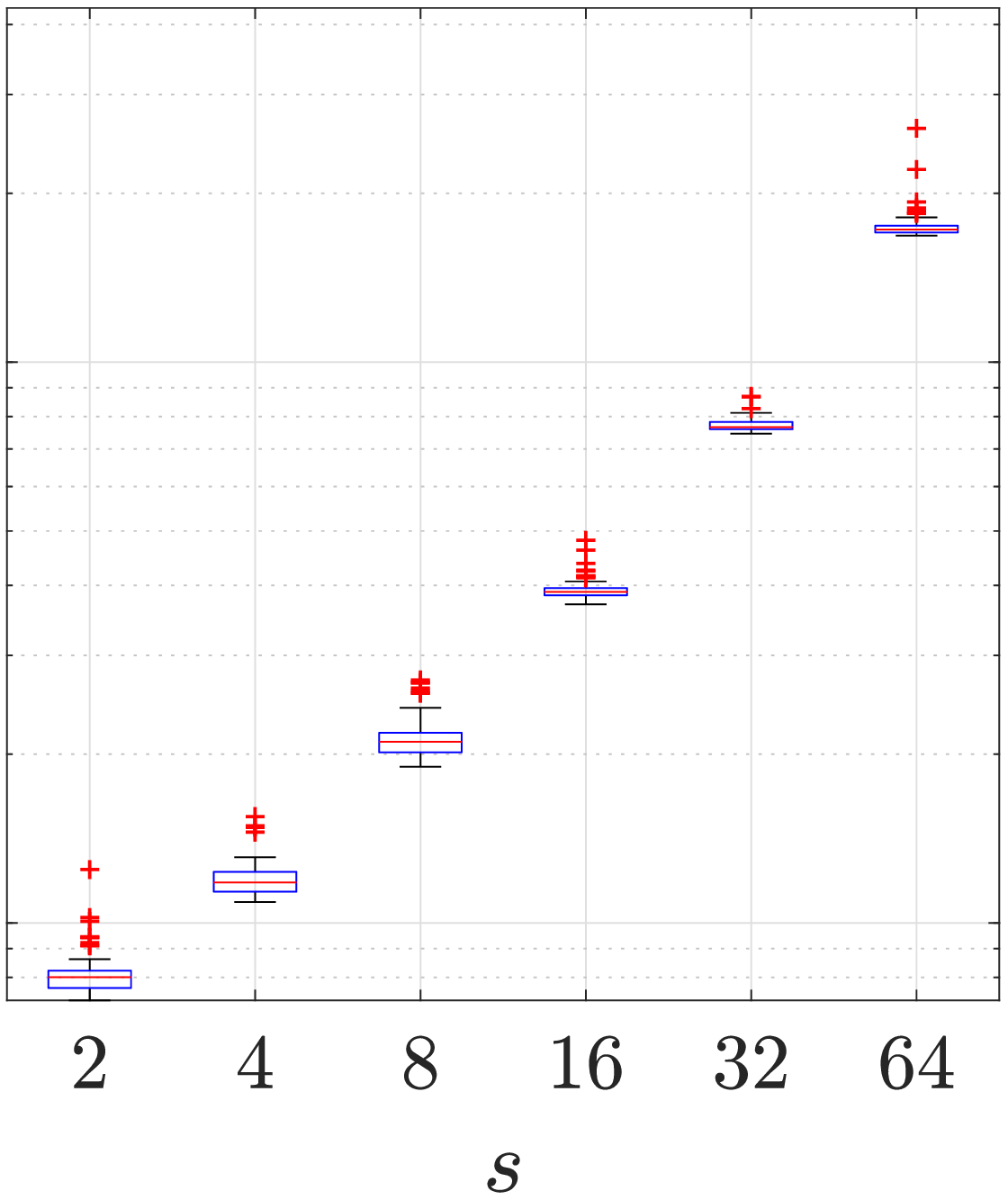} 
\end{tabular}
\caption{\label{figure4}Performance analysis of full and compressive spectral collocation from the accuracy and computational cost viewpoints for the recovery of the compressible solution \eqref{eq:exact_solution} to the diffusion equation with nonconstant coefficient $\eta$ defined by \eqref{eq:nonconstant_coefficient}. The box plots are referred to 100 random runs.}
\end{figure}
The recovery and assembly times are analogous to those of Fig.~\ref{figure2}. In terms of accuracy, we are of course not able to obtain exact recovery, as in the sparse case. The relative $L^2(\Omega)$-error associated with the full spectral approximation is $4.0 \cdot 10^{-3}$. When performing $s$ iterations of OMP on the full system (Fig.~\ref{figure4} top center), the error decays up to $s=32$, when the accuracy saturates to the level of the full approximation. The situation is analogous for the compressive approach, and the decay of the recovery error shares the same trend as the full approach with OMP recovery, up to a distortion due to randomization and to subsampling. Of course, the assembly cost is always lower for the compressive approach. The recovery cost is lower for $s \leq 16$. The values $s = 8,16,32$ seem to be realize a good trade-off between accuracy and computational efficiency.

\section{Conclusions}
\label{sec:conclusions}

We have proposed a compressive spectral collocation approach for the numerical solution of PDEs, focusing on the case of the homogeneous diffusion equation (Algorithm~\ref{alg:CSC}). 

From the theoretical viewpoint, we have shown that the proposed approach satisfies the restricted isometry property of compressive sensing under suitable assumptions on the diffusion coefficient (Theorem~\ref{thm:RIP_CSC}). This implies sparse recovery properties for the method, discussed in Section~\ref{sec:recovery_discussion}. 

From the numerical viewpoint, we have implemented the method in \textsc{Matlab$^\text{\textregistered}$} and compared it with the corresponding full spectral collocation approach in the two-dimensional case (Section~\ref{sec:numerics}). In the case of exact sparsity, the compressive method outperforms the corresponding full spectral method both in terms of accuracy and sparsity. For compressible solutions, we have studied the trade-off between accuracy and computational efficiency, showing that the compressive approach can reduce the computational cost while preserving good accuracy.

This first study shows the promising nature of the compressive spectral collocation method. However, many issues still remain open for future investigation. First, a rigorous study of the recovery guarantees of the method. Moreover, when $d \gg 1$, the approach suffers from the curse of dimensionality. This effect may be lessened by resorting to weighted $\ell^1$-minimization and by considering smaller multi-index spaces, using techniques analogous to \cite{adcock2017compressed,chkifa2017polynomial}. The method can be generalized in a straightforward way to advection-diffusion-reaction equations, but its analysis in this case deserves a careful investigation. Finally, the application of the method to nonlinear problems is also a next promising research direction.

\section{Acknowledgements}
The author acknowledges the support of the Natural Sciences and Engineering Research Council of Canada through grant number 611675 and the Pacific Institute for the Mathematical Sciences (PIMS) through the program ``PIMS Postdoctoral Training Centre in Stochastics''. Moreover, the author gratefully acknowledge Ben Adcock and the anonymous reviewer for providing helpful comments on the first version of this manuscript.

\bibliographystyle{plain}
\bibliography{biblio}

\end{document}